\documentclass[a4paper,12pt]{amsart}
\usepackage{float}
\usepackage{amsfonts,amsthm,amssymb,amsmath}
\usepackage{graphicx}
\usepackage{color}
\usepackage[all,cmtip]{xy}

\usepackage{t1enc}
\usepackage[latin1]{inputenc}

\newtheorem{theorem}{Theorem}[section]
\newtheorem{proposition}[theorem]{Proposition}
\newtheorem{corollary}[theorem]{Corollary}
\newtheorem{lemma}[theorem]{Lemma}
\newtheorem{definition}[theorem]{Definition}
\newtheorem{remark}[theorem]{Remark}

\newcommand{\N}{\mathbb{N}}
\newcommand{\U}{\mathfrak A}

\newcommand{\E}{\mathcal{E}}
\newcommand{\I}{\mathcal{I}}
\newcommand{\A}{\mathcal{A}}

\def\lg{\mathcal{L}}

\newcommand{\K}{\mathbb{K}}

\newcommand{\Nu}{\mathcal{N}}
\newcommand{\Q}{\mathcal{Q}}

\newcommand{\p}{\mathcal{P}}
\newcommand{\dps}{\displaystyle}

\begin{document}

\title[Coincidence of extendible ideals with their minimal kernel ]{Coincidence of extendible vector-valued  ideals with their minimal kernel}

\author{Daniel Galicer
\and
Rom\'an Villafa\~ne}

\thanks{The authors were partially supported by CONICET PIP 0624, PICT 2011-1456 and UBACyT 20020100100746.
 The second author has a doctoral fellowship from CONICET}

\date{}

\subjclass[2010]{46G25,46B22, 46M05, 47H60}
\keywords{Multilinear mappings, Radon-Nikodým property, Polynomial ideals, Metric theory of tensor products}

\address{Departamento de Matem\'{a}tica - Pab I,
Facultad de Cs. Exactas y Naturales, Universidad de Buenos Aires,
(C1428EGA) Buenos Aires, Argentina and IMAS -  CONICET.} \email{dgalicer@dm.uba.ar}
\email{rvillafa@dm.uba.ar}

\begin{abstract}
We provide coincidence results for vector-valued ideals of multilinear operators. More precisely, if $\U$ is an ideal of $n$-linear mappings  we give conditions for which the following equality  $\mathfrak A(E_1,\dots,E_n;F) = {\mathfrak A}^{min}(E_1,\dots,E_n;F)$ holds isometrically.  As an application, we obtain in many cases that the monomials form a Schauder basis of the space $\mathfrak A(E_1,\dots,E_n;F)$.
Several structural and geometric properties are also derived using this equality. We apply our results to the particular  case where $\mathfrak A$ is the classical ideal of extendible or Pietsch-integral multilinear operators.
Similar statements are given for ideals of vector-valued
homogeneous polynomials.

\end{abstract}

\maketitle

\section*{Introduction}

A natural question in the theory of multilinear operators, and also of homogeneous polynomials, is to find conditions under which an ideal $\U$ coincides isometrically with its minimal kernel $\U^{min}$ (see \cite{Al85a,Al85b,BoRy01,CarDi00,CarGal11,Lew77} and \cite[Section 33]{DefFlo93}, which deal with problems of this nature).
The reason for this is that, in many cases, this allows a tensorial representation of the ideal.
Frequently, the tensor product inherits many structural characteristics from those properties of the spaces involved. For example, a known result due to Gelbaum and Gil de Lamadrid \cite{GeLama61,Lama63} states that the tensor $E_1 \tilde \otimes_\alpha E_2$ has a Schauder basis if both spaces $E_1$ and $E_2$ have a basis.
This can be extended recursively to tensor products of any
number of spaces, since the tensor product is associative (see the comments before Theorem \ref{bases ideales} below).

Other properties (such as separability, Asplund  or the Radon-Nikodým properties), in many cases are also preserved by the tensor product (see for example \cite{Bu03,BuBus06,BuDieDowOja03,CarGal11,RueSte82} and also the references therein).
A tensorial representation of the ideal and these kind of transference results, permit to deduce many attributes of the space $\U(E_1,\dots,E_n;F)$. Moreover, as the elements of $\U^{min}$ may be usually approximated by finite type operators, we obtain the same property for $\U$.

Therefore, we are interested in knowing when the canonical mapping
\begin{equation*}
\xymatrix{\varrho: (E_1'\widetilde\otimes\dots\widetilde\otimes E_n'\widetilde\otimes F;\alpha)\ar@{->>}[r]^{\;\;\;\;\;\;1} & \U^{min}(E_1,\dots,E_n;F)\ar@{^{(}->}[r] & \U(E_1,\dots,E_n;F)}
\end{equation*}
results a quotient mapping or an isometric isomorphism (here $\alpha$ stands for the  tensor norm associated to $\U$).
Obviously, in both cases, we get $\U^{min}$ equals to $\U$.

Lewis (\cite{Lew77} and \cite[33.3]{DefFlo93}) obtained many results of the form $\A^{min}(E;F')=\A(E;F')$ if $\A$ is a maximal operator ideal or, in other words, coincidences of the form $E' \tilde \otimes_\alpha F' = (E \tilde \otimes_{\alpha'} F)'$.
Based on Lewis' work,  the first author and Carando tackled in \cite{CarGal11} an analogous problem for scalar ideals of multilinear operators and polynomials (i.e., where the target space is the scalar field).  As in Lewis' theorem, the Radon-Nikod\'ym property becomes a key ingredient. In this article we follow the lines of \cite{CarGal11} to address the vector-valued case. We stress that the vector-valued problem adds some technical difficulties.

For a given ideal of multilinear operators, we introduce in Definition \ref{def RNp} a vector-valued Radon-Nikodým property in the sense of \cite{Lew77, CarGal11} (where a similar property is given for tensor norms). This definition is related with a coincidence result in $c_0$-spaces.
For extendible ideals (ideals in which every multilinear operator can be extended to any superspace of the domain, see Section~\ref{preliminares}) which enjoy the latter property we prove, in the main theorem, Theorem \ref{Lewis theorem}, that $\U^{min}(E_1,\dots,E_n;F)$ coincides isometrically with $\U(E_1,\dots,E_n;F)$ for Asplund spaces $E_1, \dots, E_n$. It is noteworthy that the main theorem is not only based on what was done by Lewis and Carando-Galicer, but it generalizes both results.
Finally we relate in Theorem \ref{bases ideales}, Proposition~\ref{separabilidad} and Theorem~\ref{RN para tensores}, intrinsic attributes of $\U(E_1,\dots,E_n;F)$ with properties of $E_1,\dots, E_n, F$ and their tensor product. Namely, the existence of Schauder bases, separability, the Radon-Nikodým and Asplund properties.

We give some applications of these results for the ideals of extendible multilinear operators $\E$ and Pietsch-integral multilinear operators $P\I$.
In Corollary \ref{bases en extendibles}, we obtain that if $E_1,\dots, E_n$ are Asplund spaces and $F'$  contains no copy of $c_0$ then the canonical mapping between $\E^{min}(E_1,\dots,E_n;F')$ and $ \E(E_1,\dots,E_n;F')$ is an isometric isomorphism. Moreover, if $E_1',\dots,E_n', F'$ also have a basis, then the monomials with the square ordering form a Schauder basis of $\E(E_1,\dots,E_n;F')$.
With additional hypothesis we obtain, in Corollary~\ref{coro monomials}, similar statements in the case where the target space is not necessarily a dual space. In Corollary~\ref{coro asplund} we prove that $\E(E_1,\dots,E_n;F')$ has the Radon-Nikodým property if and only if $E_1,\dots,E_n,F$ are Asplund spaces.

A classical result due to Alencar \cite{Al85a} shows that the space of Pietsch-integral multilinear operators $P\I$ coincides isometrically, on Asplund spaces, with its minimal kernel (the space of nuclear operators, $\Nu$). We deduce this statement as a particular case of our main results. It is worth mentioning that we do not use the vector measure theory machinery \cite{DieUhl77} as Alencar did. Our perspective is completely different, it strongly relies on tensor products techniques and the fact that $P\I$ enjoys the vector-valued Radon-Nikodým property.
We obtain in Corollary \ref{coroparagrot} a coincidence result for the ideal $G\I$ of Grothendieck-integral multilinear operators as well.

We also state coincidence results for vector-valued ideals of homogeneous polynomials. In particular, we obtain in Section \ref{seccion polinomial} similar theorems for $\p_e$ and $\p_{P\I}$ (i.e., the ideals of extendible and Pietsch-integral homogeneous polynomials, respectively). For the latter ideal, we recover some known results given in \cite{BoRy01,CarDi00}.

The article is organized as follows. In Section \ref{preliminares} we state all the necessary background on ideals of multilinear operators and their associated tensor norms. We also fix some standard notation and recall basic definitions of the theory of Banach spaces. In Section \ref{main results} we prove our vector-valued Lewis type theorems (coincidence results) and some of their consequences. In Section \ref{applications} we give the mentioned applications for the ideals $\E$ and $P\I$.
Finally, in the last section we extend our results to the polynomial context.

\section{Preliminaries}\label{preliminares}

We set some notation: $E$, $F$, $G$ are real or complex Banach spaces, $E'$ is the dual space of $E$ and $J_E: E \longrightarrow E''$ is the canonical embedding of $E$ into its bidual. We denote by $B_E$ the closed unit ball of $E$ and by $FIN(E)$  the class of all finite dimensional subspaces of $E$.

A surjective linear operator $S: E \to F$ is called \emph{a metric surjection or a quotient} if $\|S(x)\|_F=\inf\{\|y\|_E : S(y)=S(x) \},$ for all $x \in E$. As usual, a mapping $I : E \to F$ is called \emph{an isometry} if $\|Ix\|_F = \|x\|_E$ for all $x \in E$. We use the notation  $\overset 1 \twoheadrightarrow$ and  $\overset 1 \hookrightarrow$ to indicate a metric surjection or an isometry, respectively.
We also write $E\overset 1=F$ whenever $E$ and $F$ are isometrically isomorphic spaces (\emph{i.e.,} there exists $I : E \to F$ a surjective isometry).

\bigskip

For Banach spaces $E_1,\dots,E_n$, we denote by $\otimes_{j=1}^n E_j$ the $n$-fold
tensor product and by $\sum_{j=1}^r \lambda_j\cdot x_1^j\otimes\dots\otimes x_n^j$ one of its elements. When the Banach spaces are vector spaces over $\mathbb C$, the scalars are not needed in the previous expression. For simplicity, we use the complex notation, although our results hold for real and complex spaces.

There is no general reference for tensor norms of order $n$ on tensor products of
Banach spaces, though one can find the definition and some
properties in \cite{FloHun02}. All abstract theory on such tensor norms that we are going to
use comes as a natural generalization of the bilinear case for which we refer
to \cite{DefFlo93}. We include some basic definitions.

We say that  $\alpha$ is a \emph{tensor norm  of order $n$} if $\alpha$ assigns to  the normed spaces $E_1,\dots,E_n$ a norm $\alpha \big(\; . \;; \otimes_{j=1}^{n} E_j \big)$ on the $n$-fold tensor product $\otimes_{j=1}^n E_j$ such that
\begin{enumerate}
\item $\varepsilon \leq \alpha \leq \pi$ on $\otimes_{j=1}^n E_j$, where $\varepsilon$ and $\pi$ are the classical injective and projective tensor norms.
\item $\|  T_1\otimes\dots\otimes T_n : (\otimes_{j=1}^{n} E_j;\alpha)   \to (\otimes_{j=1}^{n} F_j;\alpha) \| \leq \|T_1\|\cdots\|T_n\|$ for any operators $T_i \in \mathcal{L}(E_i,F_i)$ (\emph{metric mapping property}).
\end{enumerate}
We denote by $(\otimes_{j=1}^n E_j;\alpha)$ the tensor product $\otimes_{j=1}^{n}E_j$ endowed with the norm $\alpha \big(\; . \; ; \otimes_{j=1}^{n} E_j \big)$, and we write
$(\widetilde{\otimes}_{j=1}^{n}E_j;\alpha)$ for its completion.

A tensor norm $\alpha$ is \emph{finitely generated} if for every normed spaces $E_1,\dots,E_n$ and $z \in \otimes_{j=1}^{n} E_j$ we have $ \alpha (z; \otimes_{j=1}^{n}E_j) = \inf \{ \alpha(z; \otimes_{j=1}^{n}M_j) : M_j \in FIN(E_j),\; z \in M_1\otimes\dots\otimes M_n \}.$ For example, $\pi$ and $\varepsilon$ are finitely generated tensor norms.

If $\alpha$ is a tensor norm of order $n$, then the \emph{dual tensor norm $\alpha'$}\index{s-tensor norm!dual tensor norm} is defined on $FIN$ (the class of finite dimensional spaces) by $(\otimes_{j=1}^n M_j;\alpha') :\overset 1 = ( \otimes_{j=1}^n M_j';\alpha)'$
%\begin{equation*}   (\otimes_{j=1}^n M_j;\alpha') :\overset 1 = ( \otimes_{j=1}^n M_j';\alpha)'\end{equation*}
and on $NORM$ (the class of normed spaces) by
$ \alpha' ( z; \otimes_{j=1}^n E_j ) : = \inf \{ \alpha' (z; \otimes_{j=1}^n M_j ) : z \in M_1\otimes\dots\otimes M_n \},$
the infimum being taken over all of finite dimensional subspaces $M_j$ of $E_j$ whose tensor product contains $z$ (see \cite[Section 15]{DefFlo93}). By definition, $\alpha'$ is always finitely generated. It is well known that $\pi'=\varepsilon$ and $\varepsilon'=\pi$.

\bigskip

We now recall several definitions of the theory of multilinear ideals. Continuous multilinear operators are exactly those bounded on the unit ball.
The space of all continuous $n$-linear operators from $E_1\times\dots\times E_n$ to $F$ is denoted by $\lg(E_1,\dots,E_n;F)$. This class is a Banach space endowed with the norm
\[ \|T\|_{\lg(E_1,\dots,E_n;F)}=\sup_{x_i\in B_{E_i}} \|T(x_1,\dots,x_n)\|_F.\]
An \emph{ideal of continuous vector-valued $n$-linear operators} is a pair $(\U,\|\cdot\|_{\U})$ such that:
\begin{enumerate}
\item[(i)] $\U(E_1,\dots,E_n;F):=\U \cap \lg(E_1,\dots,E_n;F)$ is a linear subspace of $\lg(E_1,\dots,E_n;F)$ and $\|\cdot\|_{\U}$ is a norm which makes the pair
$(\U(E_1,\dots,E_n;F),\|\cdot\|_{\U})$ a Banach space.

\item[(ii)] If $R_j\in \mathcal{L} (X_j;E_j)$ for $1\leq j\leq n$, $T \in \U(E_1,\dots,E_n;F)$ and $S\in\lg(F;Y)$ then $S\circ T\circ (R_1,\dots,R_n)\in \U(X_1,\dots,X_n;Y)$ and

    $ \|S\circ
T\circ (R_1,\dots,R_n)\|_{\U}\le  \|S\|\cdot \|T\|_{\U}\cdot \| R_1\|\cdots\|R_n\|.$

\item[(iii)] The mapping $(\lambda_1,\dots,\lambda_n)\mapsto \lambda_1\cdots \lambda_n$ belongs to $\U(^n\mathbb K; \K)$
and has norm 1.
\end{enumerate}

Note that trivially $\lg$ is an ideal of multilinear operators. We also give the definitions of some classical multilinear operators ideals (endowed with the same norm as $\lg$) and set some usual notation.

We denote by $\lg_f$  the ideal of \textit{finite type multilinear operators}. An $n$-linear operator $T\in\lg_f(E_1,\dots,E_n;F)$ if there exist $(x_k^j)'\in E'_k$ and $f_j\in F$  such that for every $x_k$ in $E_k$ $(1\leq k\leq n)$, $$T(x_1,\dots,x_n)=\sum_{j=1}^r (x_1^j)'(x_1)\cdots (x_n^j)'(x_n)\cdot f_j.$$ The closure of the class of finite type multilinear operators in $\lg(E_1,\dots,E_n;F)$ is the ideal of \textit{approximable multilinear operators} and is denoted by $\lg_{app}(E_1,\dots,E_n;F)$. Finally, $\lg_{wsc}$ stands for the ideal of \textit{weakly sequentially continuous multilinear operators}.  Recall that $T\in\lg_{wsc}(E_1,\dots,E_n;F)$ if for every weakly convergent sequences $x_k^j \overset{w}{\longrightarrow} x_k$ in $E_k$ $(1\leq k\leq n)$, we have $T(x_1^j,\dots,x_n^j) \to T(x_1,\dots,x_n)$ in $F$.

Other well known ideals of multilinear operators that we are going to deal with are the \textit{nuclear, Pietsch integral, Grothendieck integral and extendible} ($\Nu$, $P\I$, $G\I$ and $\E$ respectively). We skip all these definitions now and leave them to Section~\ref{applications}.

\emph{The minimal kernel of $\U$} is defined as the composition ideal $\U^{min} := \overline{\mathfrak{F}}\circ\U \circ (\overline{\mathfrak{F}},\dots,\overline{\mathfrak{F}})$, where $\overline{\mathfrak{F}}$ stands for the ideal of approximable operators. In other words, a multilinear operator $T_1$ belongs to $\U^{min}(E_1,\dots,E_n;F)$ if it admits a factorization
\begin{equation} \label{factominimal}
\xymatrix{ E_1\times\dots\times E_n  \ar[rr]^{\;\;\;\;\;\;\;\;T_1} \ar[d]_{(R_1,\dots,R_n)} & & F \\
 X_1\times\dots\times X_n\ar[rr]^{\;\;\;\;\;\;\;\;T_2}  & & Y \ar[u]_S },
\end{equation}
where
%$X_1,\dots,X_n, Y$ are Banach spaces,
$S,R_1,\dots,R_n\in \overline{\mathfrak{F}}$ and $T_2\in\U(X_1,\dots,X_n;Y)$.
The $\U$-minimal norm of $T_1$ is given by $\|T_1\|_{\U^{min}} := \inf \{\|S\|\cdot  \|T_2\|_{\U}\cdot \|R_1\|\cdots\|R_n\| \}$, where the infimum runs over all possible factorizations as in~(\ref{factominimal}).

It is  important to mention a useful property of $\U^{min}$: if $E_1',\dots,E_n'$ and $F$  have the metric approximation property, then $\U^{min}(E_1,\dots,E_n;F)\overset1{\hookrightarrow}\U(E_1,\dots,E_n;F)$ and also  $\U^{min}(E_1,\dots,E_n;F)$ coincides isometrically with $\overline{\lg_f(E_1,\dots,E_n;F)}^{\|\; \cdot \;\|_{\U}}$. This and other properties of $\U^{min}$ can be found in {\cite{Flo01}}. An ideal of multilinear operators is said to be \emph{minimal} \index{minimal polynomial ideal} if $\U^{min}\overset{1}{=}\U$. For example, the ideals of nuclear and approximable multilinear operators are minimal. Moreover, $(G\I)^{min}=({P\I})^{min}=\Nu$ and $(\lg)^{min}=\lg_{app}$ (see \cite{Mu10} for the polynomial version of these equalities).

If  $\U$ is a vector-valued ideal of multilinear operators, its \emph{associated tensor norm} is the unique finitely generated tensor norm $\alpha$, of order $n+1$, satisfying $$\U(M_1,\dots, M_n;N) \overset 1 =  (M_1' \otimes \dots \otimes M_n' \otimes N;{\alpha})$$ for every finite dimensional spaces $M_1,\dots,M_n,N$. In that case we write $\U\sim\alpha$.
For example, $\lg\sim\varepsilon$, $\lg_{app}\sim\varepsilon$ , $\Nu\sim\pi$, $P\I\sim\pi$ and $G\I\sim\pi$.
Notice that $\U$ and $\U^{min}$ have the same associated tensor norm since they coincide isometrically on finite dimensional spaces.

Let $\U\sim\alpha$, the following theorem due to Floret~\cite[Theorem 4.2]{Flo01} exhibits a close relation between $(E_1' \otimes \dots \otimes E_n' \otimes F;\alpha)$ and $\U^{min}(E_1,\dots,E_n;F)$.

\begin{theorem} (Representation theorem for minimal ideals.) \label{representation theorem minimal}

Let $E_1,\dots, E_n, F$ be Banach spaces and let $\U$ be a minimal ideal with associated tensor norm $\alpha$. There is a natural quotient mapping
\begin{equation*}
(E_1' \widetilde\otimes \dots \widetilde\otimes E_n' \widetilde\otimes F;\alpha) \overset{1}{\twoheadrightarrow} \U(E_1,\dots,E_n;F)
\end{equation*}
defined on $E_1' \otimes \dots \otimes E_n' \otimes F$ by the obvious rule $$ \sum_{j=1}^r (x_1^j)' \otimes  \dots \otimes  (x_n^j)' \otimes f_j \mapsto  \sum_{j=1}^r (x_1^j)'(\cdot) \dots (x_n^j)'(\cdot) f_j.$$
\end{theorem}

Therefore, for any ideal $\U\sim\alpha$ (not necessarily minimal) we get $(E_1' \widetilde\otimes \dots \widetilde\otimes E_n' \widetilde\otimes F;\alpha) \overset{1}{\twoheadrightarrow} \U^{min}(E_1,\dots,E_n;F)$.
It should be mentioned also that, if $E_1',\dots,E_n',F$ have the bounded approximation property, then $(E_1' \widetilde\otimes \dots \widetilde\otimes E_n' \widetilde\otimes F;\alpha) \overset{1}{=} \U^{min}(E_1,\dots,E_n;F),$  as can be seen in \cite{Flo01}.

We mainly work with extendible ideals of multilinear operators. Recall that an ideal $\U$ is \textit{extendible} if the following holds:
for every Banach spaces $E_1,E_2,\dots, E_n,F$, superspaces $G_1 \supset E_1, \dots, G_n \supset E_n$ and $T \in \U(E_1, \dots, E_n;F)$, there exists an extension $\widetilde{T} \in \U(G_1, \dots, G_n ;F)$ of $T$ with the same $\U$-norm.
Some examples of extendible ideals are $P\I$ and $\E$ (this property is studied in \cite{Car99,CarZal99,KirRya98}, but for the polynomial analogues of these ideals).

We end this section refereing the reader to \cite{AlbKal06,DefFlo93,DieUhl77} for all the standard (but unexplained) definitions of Banach space theory that appear in this article.

%%%%%%%%%%%%%%%%%%%%%%%%%%%%%%%%%%%%%%%%%%%%%%%%%%%%%%%%%%%%

\section{Coincidence on ideals of multilinear operators}\label{main results}

The representation theorem for minimal ideals~\ref{representation theorem minimal} gives a natural norm one inclusion from $(E_1'\widetilde\otimes\dots\widetilde\otimes E_n'\widetilde\otimes F;\alpha)$ to $\U(E_1,\dots,E_n;F)$ defined by
\[
\xymatrix{\varrho:(E_1'\widetilde\otimes\dots\widetilde\otimes E_n'\widetilde\otimes F;\alpha)\ar@{->>}[r]^{\;\;\;\;\;\;1} & \U^{min}(E_1,\dots,E_n;F)\ar@{^{(}->}[r]^{\;\;\;\;\;\leq \ 1} & \U(E_1,\dots,E_n;F)}.
\]
An important observation is that, if $\varrho$ is a quotient mapping or an isometric isomorphism, we obtain that $\U^{min}= \U$.

To study when the mapping $\varrho$ is actually a quotient mapping a condition on the tensor norm is needed. A fundamental ingredient  both in  \cite{Lew77} and \cite{CarGal11}, where coincidence results are studied (in the operator frame and multilinear/polynomial context, respectively), is the Radon-Nikodým property for tensor norms. Based on the definitions therein, we give a vector-valued version of this notion not for tensor norms but for ideals of multilinear operators.

\begin{definition}\label{def RNp}
Let $\U \sim \alpha$ be an ideal of multilinear operators and $F$ be a Banach space. We say that $\U$ has the \textit{$F$-Radon-Nikodým property} ($F$-RNp) if $$(\ell_1(J_1) \widetilde\otimes \dots \widetilde\otimes \ell_1(J_n) \widetilde\otimes F, \alpha) \overset 1 \twoheadrightarrow  \U(c_0(J_1),\dots,c_0(J_n);F),$$
for all $J_1,\dots,J_n$ index sets.

If $\U$ has the $F$-RNp for all $F$, we say that $\U$ has the vector-RNp.
\end{definition}

 The previous definition says that if $\U$ has the $F$-RNp then $\U(c_0(J_1),\dots,c_0(J_n);F)$ coincides with the ideal $\U^{min}(c_0(J_1),\dots,c_0(J_n);F)$
for all $J_1,\dots,J_n$ index sets.

In many cases, it is enough to check this property only for countable index sets (i.e., $J_1 = \dots = J_n = \N$), as we show in the following proposition.

\begin{proposition}\label{teo cociente}
  Let $\U\sim\alpha$ be an ideal such that
  $$(\ell_1 \widetilde\otimes \dots \widetilde\otimes \ell_1 \widetilde\otimes F, \alpha) \overset 1 \twoheadrightarrow  \U(c_0,\dots,c_0;F).$$
  If $F$ contains no copy of $c_0$ or $\U\subseteq\lg_{wsc}$, then $\U$ has the $F$-RNp.
\end{proposition}
Roughly speaking, if $\U(c_0,\dots,c_0;F)$ coincides with $\U^{min}(c_0,\dots,c_0;F)$, the previous statement allows us to conclude the existence of a quotient mapping over larger $c_0$-spaces (i.e., $\U(c_0(J_1),\dots,c_0(J_n);F)=\U^{min}(c_0(J_1),\dots,c_0(J_n);F)$
for all $J_1,\dots,J_n$ index sets). We follow some ideas from \cite[Proposition 3.2]{CarGal11}. We need a couple of results first in order to give a proof.

A folklore result known as the Littlewood-Bogdanowicz-Pe{\l}czy{\'n}ski theorem (see \cite{Lit30,Bog57,Pel57} and also \cite{Di99}) states that every scalar multilinear operator $T: c_0 \times \dots\times c_0 \to \mathbb{K}$ is approximable. We present here a vector-valued version of this theorem.
More precisely, we show that, if $F$ contains no copy of $c_0$,  then $\lg(^nc_0;F)=\lg_{app}(^nc_0;F)$. We stress that this can be derived from \cite{GonGu94}. However, it is not presented in this manner and follows from a more general result. Indeed, if $F$ contains no copy of $c_0$ then every operator $S: c_0 \to F$ is weakly compact (see \cite[Theorem 2.4.10]{AlbKal06}), therefore the result follows from \cite[Theorem 6]{GonGu94} (since $c_0$ has the Dunford-Pettis property).
We give an independent proof based on elemental properties of $c_0$.

Recall that a (formal) series $\sum_{j \in \N} x_j$ in a Banach space $E$ is \textit{weakly unconditionally Cauchy} if for every $\varphi\in E'$, $\sum_{j \in \N}|\varphi(x_j)|<\infty$.

\begin{proposition}\label{bogda}
  $\lg(^nc_0;F)=\lg_{app}(^nc_0;F)$ if and only if $F$ contains no copy of $c_0$.
\end{proposition}
%In particular, this proposition says in our terminology that, $F'$ contains no copy of $c_0$ if and only if $\l$ has the $F$-RNp. We now give the proof.

\begin{proof}
If $F=c_0$ it is easy to see that there exists $T\in\lg(^nc_0;c_0)$ but not in $\lg_{app}(^nc_0;c_0)$. Indeed, take $T(x_1,\dots,x_n)= (x_1(1)\cdots x_{n-1}(1)\cdot x_n(j))_{j\in\N}.$  If $T$ is approximable, so is $T (e_1,\dots,e_1,\cdot)$, but $T(e_1,\dots,e_1,\cdot)=Id_{c_0}(\cdot)$, which is a contradiction.
The case where $F$ contains  copy of $c_0$ easily follows from this.

Conversely, suppose $F$ contains no copy of $c_0$. We show first that every  $T \in\lg(^nc_0;F)$ is weakly sequentially continuous at the origin.
Suppose there exists $T\in\lg(^nc_0;F)$ such that $T$ is not weakly sequentially continuous at the origin. By using the norm continuity of $T$, the basis of $c_0$ and by taking subsequences, if necessary, we can construct sequences $(u_j^1)_j,\dots (u_j^n)_j\subset c_0$  and a strictly increasing sequence of non-negative integers $(k_j)_j$ such that
\begin{equation}\label{mayoradelta}
  \|T(u_j^1,\dots,u_j^n)\|\geq\delta,
\end{equation}
 for some $\delta>0$, with
$$u_j^i=\sum_{l=1}^{k_{j+1}}x_{l,j}^ie_l,\ \ \ \ \ \    \|\sum_{l=1}^{k_j}x_{l,j}^ie_l\|\leq \frac{1}{2^j} \ \ \ \ \ \text{and} \ \ \ \ \ \|\sum_{l=k_j+1}^{k_{j+1}}x_{l,j}^ie_l\|=1,$$
for all $1\leq i\leq n$ and $j\in\N$. Take $v_j^i=\sum_{l=k_j+1}^{k_{j+1}}x_{l,j}^ie_l$ then, by the Bessaga-Pe{\l}czy{\'n}ski principle~\cite[Proposition 1.3.10]{AlbKal06}, $(v_j^i)_j$ is a block basis of a subspace of $c_0$ equivalent to the canonical basis of $c_0$ for each $1\leq i\leq n$.  Using \cite[Proposition 2]{Zal93} we obtain that
$\sum_{j\in\N}|\varphi(T(v_j^1,\dots,v_j^n))|\leq \|\varphi\circ T\|<\infty$ for all $\varphi\in F'$, since $\varphi\circ T\in\lg(^nc_0)$.
Thus, the formal series $\sum_{j\in\N}T(v_j^1,\dots,v_j^n)$ is weakly unconditionally Cauchy. Since $F$ contains no copy of $c_0$, this implies that the series $\sum_{j\in\N}T(v_j^1,\dots,v_j^n)$ is unconditionally convergent \cite[Theorem 2.4.11]{AlbKal06}, which yields that $\dps{\lim_{j \to \infty} T(v_j^1,\dots,v_j^n) = 0.}$

Now,
\begin{eqnarray*}
 \|T(u_j^1,\dots,u_j^n)\| & \leq & \|T(v_j^1,\dots,v_j^n)\| + \|T\|\Big(\sum_{r=1}^n \binom n r \frac{1}{2^{rj}}\Big).
 \end{eqnarray*}
Since both terms tend to zero as $j$ tends to infinity, this contradicts \eqref{mayoradelta}. Therefore, we have proved that every $n$-linear operator in $\lg(^nc_0;F)$ is weakly sequentially continuous at the origin.

Using that every $k$-linear operator in $\lg(^kc_0;F)$ is weakly sequentially continuous at the origin for every $1\leq k\leq n$, a standard argument shows that every $n$-linear operator in $\lg(^nc_0;F)$ is in fact weakly sequentially continuous.

To finish the proof it remains to see
    that every weakly sequentially continuous  $T : c_0 \times \dots \times c_0 \to F$  is approximable (i.e., can be uniformly approximated by
    finite type multilinear operators). Since $c_0$ contains no copy of $\ell_1$, by a multilinear version of \cite[Proposition 2.12]{ARHERVAL83},  we know that every weakly sequentially continuous multilinear mapping $T : c_0 \times \dots \times c_0 \to F$ is  weakly uniformly continuous on bounded sets  (i.e., $\mathcal{L}_{wsc}(^n c_0,F) = \mathcal{L}_{w}(^n c_0,F)$). It is important to remark that \cite[Proposition 2.12]{ARHERVAL83} is a result on polynomials, but follows analogously in the multilinear frame.
    On the other hand, since $c_0'=\ell_1$ has the approximation property, then $\mathcal{L}_{w}(^n c_0,F) = \mathcal{L}_{app}(^n
    c_0,F)$. This result can be found in \cite[Proposition 2.7]{ArPro80}, but again in the context of polynomial mappings.
\end{proof}

From Proposition \ref{bogda} it follows that if $F$ contains no copy of
$c_0$, then every continuous homogeneous polynomial from $c_0$ to $F$ is approximable. This provides new examples of vector-valued $\pi_1$-holomorphy types,
a class of polynomials introduced in \cite{FaJa09} and successfully explored in \cite{BeBoFaJa13}.

We are now able to prove Proposition \ref{teo cociente}.

\begin{proof}(of Proposition \ref{teo cociente}).
Let $T\in\U(c_0(J_1),\dots,c_0(J_n);F)$ and let $L=\{(j_1,\dots,j_n) : T(e_{j_1},\dots,e_{j_n})\neq0 \}$. Note that $L$ is a countable set. If not, there exist $(j_1^k,\dots,j_n^k)_{k\in\N}$ different indexes such that $|T(e_{j_1^k},\dots,e_{j_n^k})|>\varepsilon $, for some $\varepsilon>0$. Without loss of generality we can assume that the sequence of first coordinates $j_1^k$ has all its
elements pairwise different. Passing to subsequences we can also assume that $e_{j_i^k}$ is weakly null. If $F$ contains no copy of $c_0$ we can use the vector-valued Littlewood-Bogdanowicz-Pe{\l}czy{\'n}ski property of $c_0$, Proposition~\ref{bogda}, to obtain a contradiction. If $\U\subseteq\lg_{wsc}$, the claim follows immediately.

Let $\Omega_k: J_1 \times \cdots \times J_n \longrightarrow J_k$  be given by $(j_1, \dots, j_n) \mapsto j_k$, and set $L_k:= \Omega_k(L) \subset J_k$.
Consider the mapping  $\xi_k : c_0(J_k) \to c_0(L_k)$ given by
$$ (a_j)_{j \in J_k} \mapsto (a_j)_{j \in L_k}.$$
We also have the inclusion $\imath_k: c_0(L_k) \to c_0(J_k)$ defined by
$$ (a_j)_{j \in L_k} \mapsto (b_j)_{j \in J_k},$$
where $b_j$ is $a_j$ if $j \in L_k$ and zero otherwise. Note that, if we consider $\overline{T}:= T \circ (\imath_1, \dots, \imath_n)$, then $\overline T\in\U(c_0(L_1),\dots,c_0(L_n);F)$. Moreover,  $\overline T \circ(\xi_1, \dots, \xi_n)=T$ and $\|T\|_\U=\|\overline T\|_\U$. Finally, since the following diagram
$$ \xymatrix{
(\ell_1(L_1) \widetilde\otimes \dots \widetilde\otimes \ell_1(L_n) \widetilde\otimes F,\alpha) \ar@{->>}[rr]^1 \ar[d]_{\xi_1' \otimes \dots \otimes \xi_n'\otimes Id_{F}}
& & \U(c_0(L_1),\dots,c_0(L_n);F) \ar[d], & S \ar@{~>}[d] \\
(\ell_1(J_1) \widetilde\otimes \dots \widetilde\otimes \ell_1(J_n) \widetilde\otimes F,\alpha)\ar[rr] & & \U(c_0(J_1),\dots,c_0(J_n);F) & S\circ (\xi_1 ,\dots, \xi_n)},$$
is commutative, and the mapping $(\xi_1' \otimes \dots \otimes \xi_n'\otimes Id_{F})$ is an isometry because $\ell_1(L_k)$ is a 1-complemented subspace of $\ell_1(J_k)$ (via $\imath_k'$),  we get what we want.
\end{proof}

We have seen a result that permits us to know if a given ideal enjoys the vector-valued Radon-Nikodým property. To state our coincidence result for ideals of multilinear operators we first recall and give some definitions. A Banach space $E$ is an \textit{Asplund space} if every separable subspace of $E$ has separable dual. In particular, reflexive spaces and spaces that have separable duals (e.g., $c_0$) are Asplund. For more equivalences of the Asplund property and related topics see \cite{DieUhl77}.

For $1\leq k\leq n$, we define a canonical mapping, called the \textit{$k$-Arens extension of $T$},
\[Ext_k:\lg(E_1,\dots,E_n;F)\to \lg(E_1,\dots,E_{k-1},E_k'',E_{k+1},\dots,E_n;F''),\]
in the following way:
given $T\in\lg(E_1,\dots,E_n;F)$, consider $\overleftarrow{(J_F\circ T)}$ the $(n+1)$-linear form on $E_1\times\dots\times E_n\times F'$ given by $\overleftarrow{(J_F\circ T)}(x_1,\dots,x_n,y')=(J_F\circ T)(x_1,\dots,x_n)(y')$, now,
\[Ext_k(T)(x_1,\dots,x_k'',\dots,x_n)(y'):=x_k''\big (z\mapsto \overleftarrow{(J_F\circ T)}(x_1,\dots,x_{k-1},z , x_{k+1},\dots,x_n,y')\big ).\]

We say that $\U$ is an \textit{$F$-Arens stable ideal} if the mapping $$Ext_k:\U(E_1,\dots,E_n;F)\to \U(E_1,\dots,E_{k-1},E_k'',E_{k+1},\dots,E_n;F)$$ is well defined and results an isometry for all $1\leq k\leq n$. Note that the condition above says that the range of every Arens extension remains on $F$.
If $\U$ is $F$-Arens stable for every $F$, we just say that $\U$ is an \emph{Arens stable ideal}.
It is important to mention that  maximal ideals of multilinear operators are $F'$-Arens stable for every dual space $F'$ (see for example \cite[Extension Lemma 13.2]{DefFlo93}).

It is time to state our Lewis type theorem: a coincidence result for vector-valued multilinear operators.

\begin{theorem}\label{Lewis theorem}
Let $E_1, \dots, E_n$ be Asplund spaces. If $\U\sim\alpha$ is an $F$-Arens stable extendible ideal  with the $F$-RNp  then,
\begin{equation}\label{quotient}
(E_1' \widetilde\otimes \dots \widetilde\otimes E_n' \widetilde\otimes F, \alpha) \overset 1 \twoheadrightarrow \U(E_1,\dots,E_n;F).
\end{equation}
In particular, $ \U^{min}(E_1,\dots,E_n;F)\overset 1  = \U(E_1,\dots,E_n;F)$.
\end{theorem}

We need some previous results in order to give a proof of this theorem. We follow the lines of \cite{CarGal11} and include all the main steps for completeness.
We define for a given Banach space $E$ two natural mappings: $I_{E}: E\rightarrow \ell_{\infty}(B_{E'})$ and $Q_{E}:\ell_1(B_{E})\to E$, the canonical inclusion and the canonical quotient mapping respectively.

Let $\U$ be an $F$-Arens stable ideal of multilinear operators. We consider for each $1\leq k\leq n$ a natural operator $\Psi_k:\U(E_1,\dots,E_{k-1},c_0(B_{E_k'}),E_{k+1},\dots, E_n;F)\to\U(E_1,\dots,E_n;F)$ given by \[\Psi_k(T):=Ext_k(T)\circ (Id_{E_1},\dots,Id_{E_{k-1}},I_{E_k},Id_{E_{k+1}},\dots,Id_{E_n}).\]

\begin{remark}\label{diagrama conmutativo}
If $\U\sim\alpha$ is $F$-Arens stable, then we have the following commutative diagram

\xymatrix{ \big( (\widetilde\otimes_{j=1}^{k-1}E_j')  \widetilde{\otimes}  \ell_1(B_{E_k'}) \widetilde{\otimes} (\widetilde\otimes_{j=k+1}^{n}E_j')\widetilde\otimes F, \alpha \big) \ar[r] \ar[d]^{(\otimes_{j=1}^{k-1}Id_{E_j'})  \otimes  Q_{E_k'} \otimes (\otimes_{j=k+1}^{n}Id_{E_j'})\otimes Id_{F}}
& \U(E_1\dots,E_{k-1},c_0(B_{E_k'}),E_{k+1},\dots,E_n;F) \ar[d]^{\Psi_k}   \\
\big( (\widetilde\otimes_{j=1}^{k-1}E_j')  \widetilde{\otimes}  E_k' \widetilde{\otimes} (\widetilde\otimes_{j=k+1}^{n}E_j')\widetilde\otimes F, \alpha  \big) \ar[r]
& \U(E_1\dots,E_{k-1},E_k,E_{k+1},\dots,E_n;F).}
\smallskip

\end{remark}
The following proposition is crucial for our purposes.

\begin{proposition} \label{cociente}
Let $E_1, \dots, E_n, F$ be Banach spaces and $\U$ be an $F$-Arens stable extendible ideal. If $E_k$ is Asplund then $\Psi_k$ is a metric surjection.
\end{proposition}

\begin{proof} %(\emph{of Proposition \ref{cociente}.})
We prove it assuming that $k=1$ (the other cases are analogous). Notice that $\Psi_1$ has norm less than or equal to one (since $\U$ is $F$-Arens stable).

Fix $T \in\U(E_1,\dots,E_n;F) $ and $\varepsilon > 0$ and let $\widetilde{T} \in \U(\ell_\infty(B_{E_1'}),E_2,\dots,E_n;F)$ be an extension of $T$ with the same $\U$-norm.
Since $E_1'$ has the Radon-Nikodým property, by the Lewis-Stegall Theorem  \cite[33.1]{DefFlo93}, the adjoint of the canonical inclusion $I_{E_1}: E_1 \to \ell_{\infty}(B_{E_1'})$ factors through $\ell_1(B_{E_1'})$ via
\begin{equation*} \label{factorizacion}
\xymatrix{ {\ell_{\infty}(B_{E_1'})'}  \ar[rr]^{I_{E_1}'} \ar[rd]^{S} & & {E_1'} \\
& {\ell_1(B_{E_1'})} \ar@{->>}[ru]^{Q_{E_1'}} & }
\end{equation*}
with $\|S\| \leq (1+\varepsilon)$.
Let $B:c_0(B_{E_1'}) \times E_2\times \dots  \times E_n \to F$ be given by the formula
$B (a,x_2, \dots, x_{n})= Ext_1(\widetilde{T})(S'J_{c_0(B_{E_1'})}(a),x_2, \dots, x_{n})$. Note that $B$ is well defined. Using the ideal property and the fact that $\U$  is $F$-Arens stable we have that $B \in\U(c_0(B_{E_1'}),E_2,\dots,E_n;F) $ and $\|B\|_{\U} \leq \|T\|_\U (1+\varepsilon)$.

If we show that $\Psi_1 (B)=T$ we are done. It is an
easy exercise to prove the equality $I_{E_1}(x) (a) = Q_{E_1'}(a)(x)$ for $x \in
E_1$ and $a \in \ell_1(B_{E_1'})$. Now notice that
$$\overleftarrow{(J_F\circ B)}(\cdot,x_2,\dots,x_n,y')=S[b\mapsto y'\big( \widetilde T(b,x_2,\dots,x_n)\big)],$$
where  $[b\mapsto y'\big(\widetilde T(b,x_2,\dots,x_n)\big)]$ is a functional defined on $\ell_\infty(B_{E_1'})$.
Indeed, given $a\in c_0(B_{E_1'})$,
\begin{align*}
  S[b\mapsto y'\big(\widetilde T(b,x_2,\dots,x_n)\big)](a) & =  J_{c_0(B_{E_1'})}(a)\big ( S[b\mapsto y'\big( \widetilde T(b,x_2,\dots,x_n)\big)]\big )\\
  & = S'  J_{c_0(B_{E_1'})}(a)[b\mapsto y'\big(\widetilde T(b,x_2,\dots,x_n)\big)]\\
  & = [Ext_1(\widetilde T)(S'  J_{c_0(B_{E_1'})}(a),x_2,\dots,x_n)](y')\\
  & = y'\big(B(a,x_2,\dots,x_n)\big)\\
  & = \overleftarrow{(J_F\circ B)}(\cdot,x_2,\dots,x_n,y')(a).
\end{align*}
Finally, let $y'\in F'$, then
\begin{align*}
y'\big(\Psi_1 (B)(x_1, \dots, x_n)\big)  & = y'\big(Ext_1(B)(I_{E_1}(x_1), x_2,\dots,x_n)\big)\\
& = Ext_1(B)(I_{E_1}(x_1), x_2,\dots,x_n)(y')\\
& = (I_{E_1}(x_1)) [\overleftarrow{(J_F\circ B)}(\cdot,x_2, \dots,x_{n},y')] \\
& = Q_{E_1'} (\overleftarrow{(J_F\circ B)}(\cdot,x_2, \dots,x_{n},y') (x_1) \\
& = \big (Q_{E_1'} S [b\mapsto  y'\big(\widetilde T(b,x_2,\dots,x_n)\big)]\big ) (x_1) \\
& = \big (I_{E_1}' [b\mapsto y'\big(\widetilde T(b,x_2,\dots,x_n)\big)]\big ) (x_1)  \\
& = y'\big(\widetilde T(I_{E_1} x_1,x_2, \dots, x_n)\big) \\
& = y'\big( T(x_1, \dots, x_n)\big),
\end{align*}
which ends the proof.
\end{proof}

We are now ready to prove our main result, Theorem \ref{Lewis theorem}.
\begin{proof}{(of Theorem \ref{Lewis theorem})}

It is not difficult to see that an ideal of multilinear operators $\U$ is extendible if and only if its associated tensor norm $\alpha$ is projective in the first $n$ coordinates (as can be seen in \cite[Chapter 3]{Gal12}). In other words, if $q_1:X_1 \overset 1 \twoheadrightarrow Y_1$, $\dots$, $q_n: X_n \overset 1 \twoheadrightarrow Y_n$ are quotient mappings then the mapping
$$ q_1 \otimes \dots q_n \otimes id_Z : (X_1 \tilde \otimes \dots \tilde \otimes X_n \tilde \otimes Z, \alpha) \overset 1 \twoheadrightarrow (Y_1 \tilde \otimes \dots \tilde \otimes Y_n \tilde \otimes Z, \alpha)$$
is also a quotient.
Therefore the down arrows on the left side of the Figure \ref{diagrama}  are all quotient mappings.

\begin{figure}[H]
  \begin{equation*}
\xymatrix{ \big( \widetilde{\otimes}_{i=1}^n \ell_1(B_{E_i'})\widetilde{\otimes} F, \alpha \big) \ar[rr]^{\varrho_0} \ar@{->>}[d]^{{\otimes}_{i=1}^{n-1} Id_{\ell_1(B_{E_i'})}{\otimes} Q_{E_n'}{\otimes}Id_{F} }
&  & \U(c_0(B_{E_1'}),\dots,c_0(B_{E_n'});F) \ar@{->>}[d]_{\Psi_n} \\
( (\widetilde{\otimes}_{i=1}^{n-1} \ell_1(B_{E_i'}) ) \widetilde{\otimes} E_n'\widetilde{\otimes} F, \alpha ) \ar[rr]^{\varrho_1} \ar@{->>}[d]^{({\otimes}_{i=1}^{n-2} Id_{\ell_1(B_{E_i'})} ) {\otimes} Q_{E_{n-1}'} {\otimes} Id_{E_n'}{\otimes}Id_{F} }
& & \U(c_0(B_{E_1'}),\dots,c_0(B_{E_{n-1}'}),E_n;F) \ar@{->>}[d]_{\Psi_{n-1}} \\
((\widetilde{\otimes}_{i=1}^{n-2} \ell_1(B_{E_i'}) ) \widetilde{\otimes} E_{n-1}' \widetilde{\otimes} E_n'\widetilde{\otimes} F, \alpha ) \ar[rr]^{\varrho_2} \ar@{->>}[d]
& & \U(c_0(B_{E_1'}),\dots,c_0(B_{E_{n-2}'}),E_{n-1},E_n;F) \ar@{->>}[d] \\
\vdots & \dots & \vdots \\
\ar@{->>}[d] & \dots & \ar@{->>}[d]\\
(\ell_1(B_{E_1'}) \widetilde{\otimes} (\widetilde{\otimes}_{i=2}^{n} E_i' )\widetilde{\otimes}F, \alpha ) \ar[rr]^{\varrho_{n-1}} \ar@{->>}[d]^{Q_{E_1'} \otimes ({\otimes}_{i=2}^{n} Id_{E_i'} )\otimes Id_{F}}
& & \U(c_0(B_{E_1'}),E_2,\dots,E_n;F)  \ar@{->>}[d]_{\Psi_1} \\
(\widetilde{\otimes}_{i=1}^{n} E_i'\widetilde{\otimes} F, \alpha) \ar[rr]^{\varrho_n}
& & \U(E_1,\dots,E_n;F)  \\
}
\end{equation*}
\caption{\textbf{Commutative diagram used in the proof of Theorem~\ref{Lewis theorem}}}
\label{diagrama}
\end{figure}

On the other hand, since $E_1,\dots,E_n$ are Asplund spaces (and, obviously, also the spaces $c_0(B_{E_1'}),\dots, c_0(B_{E_n'})$) and $\U$ is $F$-Arens stable we have, by Proposition \ref{cociente}, that the down arrows on the right side are quotient mappings too.
%\newpage

Now, since $\U$ has the $F$-RNp, this implies that $\varrho_0$ is a quotient mapping and this yields that $\varrho_1$ is a quotient mapping also (by Remark \ref{diagrama conmutativo} the diagram is commutative). Proceeding inductively in each square we obtain that $\varrho_n$ is a quotient mapping and this concludes the proof.
\end{proof}

Our result, Theorem~\ref{Lewis theorem}, generalizes Lewis' theorem \cite[Theorem 33.3]{DefFlo93} and also the result given in \cite[Theorem 3.5.]{CarGal11} for scalar multilinear operators.
To see the first assertion, suppose that $\alpha$ is a tensor norm of order $2$ with the Radon-Nikod\'ym property (in the sense of \cite[33.2]{DefFlo93}) associated to the maximal operator ideal $\mathcal A$,  and let $F$ be an Asplund Banach space. We must show that Theorem~\ref{Lewis theorem} gives the quotient mapping
\[E' \widetilde\otimes_{\alpha /} F'\twoheadrightarrow (E\otimes_{(\alpha /)'}F)'.\]

Recall that the transposed tensor norm $({\alpha/})^t = {\backslash(\alpha^t)}$ is associated to the operator ideal ${\backslash(\mathcal A^{dual})}$, as can be seen in \cite[17.8 and 20.12]{DefFlo93}.

Now, since $\alpha$ has the the Radon-Nikod\'ym property the same holds for $\alpha /$ \cite[Proposition 33.2]{DefFlo93}. By Lemma 33.3 in \cite{DefFlo93} we have the quotient mapping
\[ E' \widetilde\otimes_{\alpha /} \ell_1(J) \twoheadrightarrow (E\otimes_{(\alpha /)'}c_0(J))',\]
for any index set $J$.
In other words,
\[ \ell_1(J) \widetilde\otimes_{(\alpha /)^t} E' = \ell_1(J) \widetilde\otimes_{\backslash (\alpha^t )} E' \twoheadrightarrow (c_0(J)\otimes_{(\backslash (\alpha^t ))'}E)' = {\backslash(\mathcal A^{dual})}(c_0(J);E'),\]

This implies that the ideal ${\backslash(\mathcal A^{dual})}$ has the vector-RNp in the sense of Definition~\ref{def RNp}. Note that the ideal ${\backslash(\mathcal A^{dual})}$ is $E'$-Arens stable ($E'$ is a dual space) and is extendible in the first variable, which means that any $S \in {\backslash(\mathcal A^{dual})}(X;Y)$ can be extended to an operator in ${\backslash(\mathcal A^{dual})}(Z;Y)$ with the same ideal norm, for every superspace $Z \supset X$. Therefore by Theorem~\ref{Lewis theorem} we obtain the quotient mapping
\[F \widetilde\otimes_{\backslash (\alpha^t )} E' \twoheadrightarrow  {\backslash(\mathcal A^{dual})}(F;E'),\]
for every Asplund space $F$. If we transposed this relation we get
\[E' \widetilde\otimes_{\alpha /} F'\twoheadrightarrow \mathcal A / (E;F')=(E\otimes_{(\alpha /)'}F)',\]
for every Asplund space $F$. This is exactly the result that appears in \cite[Theorem 33.3]{DefFlo93}.

To deduce \cite[Theorem 3.5]{CarGa11b} from Theorem~\ref{Lewis theorem}, the reasoning is quite similar. Moreover, we do not have to change the order of the spaces involved (as we did before).

In many cases, for an arbitrary space $F$, the ideal $\U$ is $F'$-Arens stable but not $F$-Arens stable (for example for $\E$ and $G\I$, see Section \ref{applications} for definitions). In this situation, Theorem~\ref{Lewis theorem} give us a coincidence result only in the cases where the target space is a dual.  One of the assumptions of the next theorem asks for $\U$ to be $F''$-Arens stable (a much weaker condition). As we are interested in searching for monomial basis on spaces of multilinear operators, it is natural to deal with spaces which have shrinking Schauder bases.

\begin{theorem}\label{teo achicante}
  Let $\U$ be an $F''$-Arens stable extendible ideal with the $F''$-RNp. If $E_1,\dots,E_n$ have shrinking bases and  $F''$ has the bounded approximation property, then
  \[(E_1'\tilde\otimes\dots\tilde\otimes E_n'\tilde\otimes F,\alpha)\overset 1 =\U^{min}(E_1,\dots,E_n;F)\overset 1 =\U(E_1,\dots,E_n;F).\]
\end{theorem}

This theorem will be used in the next section for a result on Schauder bases for the ideal $\E$.
We need first some lemmas  in order to give a proof of this theorem. The following asserts, under approximation properties, that finite type multilinear operators have the same $\U^{min}$-norm when the range space is $F$ or its bidual $F''$.

%What can we say when the range space is not a dual space? Let us see that with the additional hypothesis that $E_1,\dots,E_n$ has shrinking basis and $F''$ contains no copy of $c_0$ we can conclude that $\U^{min}(E,\dots,E_n;F)=\U(E_1,\dots,E_n;F)$.

\begin{lemma}\label{lemma1}
  Let $\U\sim\alpha$ be an ideal of multilinear operators and  let $E_1',\dots,E_n', F''$ be Banach spaces with the bounded approximation property. If $T\in\U(E_1,\dots,E_n;F)$ is a finite type multilinear operator, then   $\|J_F\circ T\|_{\U^{min}(E_1,\dots,E_n;F'')}=\|T\|_{\U^{min}(E_1,\dots,E_n;F)}$.
\end{lemma}
\begin{proof} Using the Embedding Lemma \cite[13.3]{DefFlo93} and the fact that $E_1',\dots,E_n',F''$ have the bounded approximation property, we have the following commutative diagram

  \[  \xymatrix{
  (\widetilde{\otimes}_{i=1}^n E_i'\widetilde{\otimes} F'', \alpha )  \ar@{=}[r] & \U^{min}(E_1,\dots,E_n;F'')\\
  (\widetilde{\otimes}_{i=1}^n E_i'\widetilde{\otimes} F, \alpha ) \ar@{=}[r] \ar@{^{(}->}[u]^1
  & \U^{min}(E_1,\dots,E_n;F)\ar@{^{(}->}[u]}, \]
  the proof easily follows from this.
\end{proof}

Given an operator in $\U^{min}(E_1,\dots,E_n;F)$, Lemma \ref{lemaproyeccion} below, shows an explicit sequence of finite type multilinear operators that approximate it.

\begin{lemma}\label{lemaproyeccion}
  Let $E_1,\dots,E_n$ be Banach spaces with shrinking bases and $F$ be a Banach space. Let $\overline{P_k}:=(P_k^1,\dots,P_k^n)$, where $P_k^j$ is the projection on the first $k$ coordinates of the basis of $E_j$. If  $T\in\U^{min}(E_1,\dots,E_n;F)$, then $T\circ \overline{P_k}\longrightarrow T$ in the $\U^{min}$-norm.
\end{lemma}

\begin{proof}
  First observe that $\U^{min}(E_1,\dots,E_n;F)$ is the closure of the finite type multilinear operators in $\U^{min}$-norm (see \cite[Lemma 3.3]{Flo01} for a similar result in the polynomial context). Fix $\varepsilon>0$, and let $R\in\U(E_1,\dots,E_n;F)$ be a finite type multilinear operator such that $\|T-R\|_{\U^{min}(E_1,\dots,E_n;F)}<\varepsilon$, then $\|R\circ \overline{P_k}-T\circ \overline{P_k}\|_{\U^{min}}\leq\|P_k^1\|\cdots\|P_k^n\|\  \varepsilon<C\varepsilon$. If $R=\sum_{j=1}^r \varphi_j^1(\cdot)\cdots\varphi_j^n(\cdot)f_j$, then
\begin{align*}
& \| R  -R\circ  \overline{P_k}\|_{\U^{min}} = \\
  & =   \left\|\sum_{j=1}^r\big( \varphi_j^1(\cdot)\cdots\varphi_j^n(\cdot)- (\varphi_j^1\circ P_k^1)(\cdot)\cdots(\varphi_j^n\circ P_k^n)(\cdot)\big)f_j\right\|_{\U^{min}}\\
  & \leq  \sum_{j=1}^r\|\big( \varphi_j^1(\cdot)\cdots\varphi_j^n(\cdot)- (\varphi_j^1\circ P_k^1)(\cdot)\cdots(\varphi_j^n\circ P_k^n)(\cdot)\big) f_j\|_{\U^{min}}\\
  & \leq  \sum_{j=1}^r(\| \big( \varphi_j^1(\cdot)\cdots\varphi_j^n(\cdot)- (\varphi_j^1\circ P_k^1)(\cdot)\varphi_j^2(\cdot)\cdots\varphi_j^n(\cdot)\big) f_j\|_{\U^{min}}+\cdots\\
  &  \;\;\;\; \cdots + \| \big( (\varphi_j^1\circ P_k^1)(\cdot)\cdots(\varphi_j^{n-1}\circ P_k^{n-1})(\cdot)\varphi_j^n(\cdot)- (\varphi_j^1\circ P_k^1)(\cdot)\cdots(\varphi_j^n\circ P_k^n)(\cdot)\big) f_j\|_{\U^{min}})\\
  & \leq    \sum_{j=1}^rD \big{(}\|\varphi_j^1(\cdot)-(\varphi_j^1\circ P_k^1)(\cdot)\|_{E_1'}+\cdots+\|\varphi_j^n(\cdot)-(\varphi_j^n\circ P_k^n)(\cdot)\|_{E_n'}\big{)}\cdot\|f_j\|_{F},
  \end{align*}
  which is less than $\varepsilon$ for $k$ sufficiently large.

  Therefore,
    \[\|T-T\circ \overline{P_k}\|_{\U^{min}}\leq \|T-R\|_{\U^{min}}+\|R-R\circ \overline{P_k}\|_{\U^{min}}+\|R\circ \overline{P_k}-T\circ \overline{P_k}\|_{\U^{min}}<K\varepsilon\]
  for $k$ big enough.
\end{proof}

Now we are ready to prove the theorem.

\begin{proof}{(of Theorem \ref{teo achicante})}

Let $T\in\U(E_1,\dots,E_n;F)$. By Theorem \ref{Lewis theorem} we have that $J_F\circ T\in\U^{min}(E_1,\dots,E_n;F'')$ (spaces with shrinking bases are Asplund spaces) and $\|J_F\circ T\|_{\U^{min}}=\|J_F\circ T\|_{\U}$. By Lemma~\ref{lemaproyeccion}, $J_F\circ T\circ\overline{P_k}\longrightarrow J_F\circ T$ in the $\U^{min}$-norm, thus $(J_F\circ T\circ\overline{P_k})_k$ is a Cauchy sequence in $\U^{min}(E_1,\dots,E_n;F'')$. By Lemma \ref{lemma1}, $\|J_F\circ T\circ\overline{P_k}\|_{\U^{min}}=\| T\circ\overline{P_k}\|_{\U^{min}}$. Moreover,  $T\circ\overline{P_k}$ is also a Cauchy sequence in $\U^{min}(E_1,\dots,E_n;F)$ which obviously converges to $T\in\U^{min}(E_1,\dots,E_n;F)$. Then,
\[\|T\|_\U \leq \|T\|_{\U^{min}}=\|J_F\circ T\|_{\U^{min}}=\|J_F\circ T\|_{\U}\leq\|T\|_\U.\]
This completes the proof.
\end{proof}

Let $E_1, \dots, E_n$ be Banach spaces with Schauder basis $(e_{j_1})_{j_1}, \dots, (e_{j_1})_{j_n}$ respectively and let $\beta$ be tensor norm of order $n$. There is a natural ordering, usually called the generalized square ordering of Gelbaum-Gil de Lamadrid (or simply square ordering) in $\mathbb{N}^n$ such that the monomials $(e_{j_1} \otimes \dots \otimes e_{j_n})_{(j_1,\dots,j_n) \in \mathbb{N}^n}$ (with this ordering) form a Schauder basis  of the tensor $(E_1 \tilde\otimes \dots \tilde\otimes E_n , \beta)$ (see for example \cite{DimZal96,GreRya05} for an appropriate treatment of  this ordering). This follows by mimicking the ideas of \cite{GeLama61} (see also \cite[Exercise 12.9]{DefFlo93}) for tensor products of order two. In \cite[Theorem 8]{CarLa08} this result is generalized for atomic decompositions.
Using our previous results we can provide conditions to ensure the existence of monomial bases on ideals of multilinear operators.

\begin{theorem}\label{bases ideales}
  Let $\U\sim\alpha$ be an extendible ideal of multilinear operators.
  \begin{itemize}
    \item [(1)] If $\U$ is $F$-Arens stable, has the $F$-RNp and $E_1',\dots,E_n', F$ have Schauder bases $(e_{j_1}')_{j_1},\dots, (e_{j_n}')_{j_n}, (f_l)_l$ respectively, then the monomials
             \[\big(\ e_{j_1}'(\ \cdot \ )\cdots e_{j_n}'(\ \cdot \ )\cdot f_l\ \big)_{j_1,\dots,j_n,l}\]
         with the square ordering form a Schauder basis of $\U(E_1,\dots,E_n;F)$.
    \item [(2)] If $\U$ is $F''$-Arens stable and has the $F''$-RNp, $F''$ has the bounded approximation property, $E_1,\dots,E_n$ have shrinking Schauder bases $(e_{j_1})_{j_1},\dots ,(e_{j_n})_{j_n}$ respectively and $F$ has basis $(f_l)_l$, then the monomials (associated to the coordinate functionals)
        \[\big(\ e_{j_1}'(\ \cdot \ )\cdots e_{j_n}'(\ \cdot \ )\cdot f_l\ \big)_{j_1,\dots,j_n,l}\]
        with the square ordering form a Schauder basis of $\U(E_1,\dots,E_n;F)$.
  \end{itemize}
\end{theorem}

\begin{proof}
Item $(1)$ is a straightforward application of Theorem \ref{Lewis theorem} and the fact that the monomials with the square ordering form a basis of the tensor product $(E_1'\tilde\otimes\dots\tilde\otimes E_n'\tilde\otimes F, \alpha)$. Recall that the quotient mapping  given in equation \eqref{quotient} in Theorem~\ref{Lewis theorem} is actually an isometric isomorphism in this case since all the spaces involved have the bounded approximation property

Item $(2)$ follows identically using Theorem \ref{teo achicante}.
\end{proof}

We now relate structural properties of $\U(E_1,\dots ,  E_n, F')$ with properties of  $E_1,\dots, E_n, F$ and their tensor product. Namely, separability, the Radon-Nikodým and Asplund properties.

The following proposition is straightforward.
\begin{proposition} \label{separabilidad}
Let $\U\sim\alpha$ be an $F$-Arens stable extendible ideal with the $F$-RNp and let $E_1, \dots, E_n, F$ be Banach spaces such that $E_i'$ and $F$ are separable spaces, for all  $1 \leq i \leq n$.
Then the space $\U(E_1,\dots,E_n;F)$ is separable.
\end{proposition}

Recall that a Banach space $E$ has the \emph{Radon-Nikod\'ym property}\index{Radon-Nikod\'ym property} if, for every finite measure $\mu$, every operator $T : L_1(\mu) \to E$ is \emph{representable}, \emph{i.e.,} there exists a bounded $\mu$-measurable function $g : \Omega \to E$ with $$Tf = \int fg d\mu \mbox{\;\; for all $f \in L_1(\mu)$.}$$
A Banach space  $E$ is Asplund if its dual $E'$ has the Radon-Nikod\'ym property.
The following theorem is a transfer type result. It shows that, under certain conditions, the Asplund property of the spaces involved can be transferred to the tensor product.
A result of this nature can be found in \cite{RueSte82} for the injective tensor product.

\begin{theorem} \label{RN para tensores}
Let $\U\sim\alpha$ be an extendible maximal ideal with the $F'$-RNp for every separable dual space $F'$ and $E_1, \dots, E_n$ be Banach spaces.
The following are equivalent:
\begin{itemize}
\item [(1)] The spaces $E_1,\dots, E_n, F$ are Asplund.
\item [(2)] The space  $(E_1\tilde\otimes\dots\tilde\otimes E_n\tilde\otimes F,\alpha')$ is Asplund.
\item[(3)] The space $\U(E_1,\dots ,  E_n, F')$ has the Radon-Nikodým property.
\end{itemize}

\end{theorem}
We give first some elementary lemmas of the theory of tensor products and tensor norms.
The first one states the if an element $z$  of the tensor product $(\tilde\otimes_{j=1}^nE_j,\alpha)$ can be approximated by finite type tensors (i.e., elements of the algebraic tensor product $\otimes_{j=1}^nE_j$) then $z$ belongs to the tensor product of some separable subspaces of $E_j$.

\begin{lemma}
   Let $\alpha$ be a finitely generated tensor norm and let $z\in(\tilde\otimes_{j=1}^nE_j,\alpha)$. Consider a sequence of finite type tensors $(w_r)_r$ such that $w_r\to z$ in $(\tilde\otimes_{j=1}^nE_j,\alpha)$. Then there exist separable subspaces $W_j\subset E_j$ $(1\leq j\leq n)$ such that:

  \begin{enumerate}
  \item $z\in(\tilde\otimes_{j=1}^n W_j,\alpha)$
    \item $\alpha(w_r;\tilde\otimes_{j=1}^n W_j)=\alpha(w_r;\tilde\otimes_{j=1}^n E_j)$, for all $r\in\N$.
    \item $\alpha(w_r-w_l;\tilde\otimes_{j=1}^n W_j)=\alpha(w_r-w_l;\tilde\otimes_{j=1}^n E_j)$, for all $r,l\in\N$.

  \end{enumerate}

\end{lemma}

\begin{proof}
Since $\alpha$ is a finitely generated tensor norm, given $k,r\in\N$, there exists $A_j^{k,r}\in FIN(E_j)$ such that $w_r\in\tilde\otimes_{j=1}^n A_j^{k,r}$ and $\alpha(w_r;\tilde\otimes_{j=1}^n A_j^{k,r})\leq\alpha(w_r;\tilde\otimes_{j=1}^n E_j)+1/k$.

In the same way, given  $k,r,l\in\N$, there exists $B_j^{k,r,l}\in FIN(E_j)$ such that $w_r-w_l\in\tilde\otimes_{j=1}^n B_j^{k,r,l}$ and $\alpha(w_r-w_l;\tilde\otimes_{j=1}^n B_j^{k,r,l})\leq\alpha(w_r-w_l;\tilde\otimes_{j=1}^n E_j)+1/k$.

Then for $1\leq j\leq n$, we take
\[W_j:=\overline{span[A_j^{k,r}, B_j^{k,r,l} : k,r,l\in\N]},\]
which are separable since $A_j^{k,r}, B_j^{k,r,l}\in FIN(E_j)$. Then,
\[\alpha(w_r;\tilde\otimes_{j=1}^n W_j)\leq\alpha(w_r;\tilde\otimes_{j=1}^n A_j^{k,r})\leq\alpha(w_r;\tilde\otimes_{j=1}^n E_j)+1/k\leq\alpha(w_r;\tilde\otimes_{j=1}^n W_j)+1/k\] for all $k\in\N$. Thus $\alpha(w_r;\tilde\otimes_{j=1}^n W_j)=\alpha(w_r;\tilde\otimes_{j=1}^n E_j)$.

Analogously, passing through $B_j^{k,r,l}$, we have $\alpha(w_r-w_l;\tilde\otimes_{j=1}^n W_j)=\alpha(w_r-w_l;\tilde\otimes_{j=1}^n E_j)$.

Now, since $w_r\to z$, $(w_r)_r$ is a Cauchy sequence in $(\tilde\otimes_{j=1}^nE_j,\alpha)$ . Then $(w_r)_r$ is Cauchy in $(\tilde\otimes_{j=1}^nW_j,\alpha)$, therefore there exists $w\in(\tilde\otimes_{j=1}^nW_j,\alpha)$ such that $w_r\to w$. We affirm that $w=z$.
Indeed,
\begin{align*}
  \alpha(z-w;\tilde\otimes_{j=1}^nE_j) & \leq\alpha(z-w_r;\tilde\otimes_{j=1}^n E_j)+\alpha(w_r-w;\tilde\otimes_{j=1}^n E_j)\\
  & \leq\alpha(z-w_r;\tilde\otimes_{j=1}^n E_j)+\alpha(w_r-w;\tilde\otimes_{j=1}^n W_j)\to0,
\end{align*}
which concludes the proof.
\end{proof}

The next result asserts that a separable subspace of the tensor product $(\tilde\otimes_{j=1}^nE_j,\alpha)$ can be isometrically embedded in the tensor product of separable subspaces of $E_j$.
This was used, for example in \cite[Theorem 2.9]{CarGal11}, however  everything was much easier there since the tensor norm considered was injective (i.e., it respects subspaces isometrically).

\begin{lemma}
Let $\alpha$ be a finitely generated tensor norm and let $S\subseteq(\tilde\otimes_{j=1}^nE_j,\alpha)$ be a separable subspace. Then exist separable subspaces $W_j\subseteq E_j$ such that $S\overset 1 \hookrightarrow(\tilde\otimes_{j=1}^nW_j,\alpha)$.
\end{lemma}

\begin{proof}

Let $\{z_k\}_k\subseteq S$ be a dense  subset. For each $k\in\N$, consider a sequence of finite type tensors $(w_r^k)_r$ such that $w_r^k\to z_k$ in $(\tilde\otimes_{j=1}^nE_j,\alpha)$. Then, by the previous lemma, there exist separable subspaces $W_j^k\subseteq E_j$ such that $z_k,w_r^k \in(\tilde\otimes_{j=1}^nW_j^k,\alpha)$ for every $k$ and $r$. Moreover, we have $\alpha(z_k;\tilde\otimes_{j=1}^n W_j^k)=\alpha(z_k;\tilde\otimes_{j=1}^n E_j)$, $\alpha(w_r^k;\tilde\otimes_{j=1}^n W_j^k)=\alpha(w_r^k;\tilde\otimes_{j=1}^n E_j)$ and $\alpha(w_r^k-w_l^k;\tilde\otimes_{j=1}^n W_j^k)=\alpha(w_r^k-w_l^k;\tilde\otimes_{j=1}^n E_j)$, for every $k,r,l$.

Take $W_j:=\overline{span[W_j^k : k\in\N]}$ which is also separable. Furthermore, $z_k\in(\tilde\otimes_{j=1}^nW_j,\alpha)$ for all $k\in\N$. Let us see $S\overset 1 \hookrightarrow(\tilde\otimes_{j=1}^nW_j,\alpha)$.
Fix $v \in S$; without loss of generality we can suppose that $\alpha(v - z_k;\tilde\otimes_{j=1}^n E_j) \to 0$, as $k \to \infty$.
Let us see that $\{z_k\}_k$ is a Cauchy sequence in $(\tilde\otimes_{j=1}^nW_j,\alpha)$. Indeed, for every $r$ we have
\begin{align*}
\spadesuit  \;\; \alpha(z_k-z_l;\tilde\otimes_{j=1}^nW_j) & \leq \alpha(z_k-w_r^k;\tilde\otimes_{j=1}^n W_j)+\alpha(w_r^k-w_r^l;\tilde\otimes_{j=1}^n W_j) + \alpha(w_r^l-z_l;\tilde\otimes_{j=1}^n W_j) \\
& = \alpha(z_k-w_r^k;\tilde\otimes_{j=1}^n W_j)+\alpha(w_r^k-w_r^l;\tilde\otimes_{j=1}^n E_j) + \alpha(w_r^l-z_l;\tilde\otimes_{j=1}^n W_j).
\end{align*}
Using the triangle inequality note that the second term $\alpha(w_r^k-w_r^l;\tilde\otimes_{j=1}^n E_j)$ of the last inequality can be bounded by $\alpha(w_r^k-z_k;\tilde\otimes_{j=1}^n E_j) + \alpha(z_k-z_l;\tilde\otimes_{j=1}^n E_j) + \alpha(z_l-w_r^l;\tilde\otimes_{j=1}^n E_j).$
Therefore $\alpha(z_k-z_l;\tilde\otimes_{j=1}^nW_j)$ in $\spadesuit$ goes to zero if $k,l$ are arbitrarily large (all the inequalities hold for every $r$).
Since $\widetilde\otimes_{j=1}^nW_j$ is complete $z_k$ converge (and reasoning as we did in the proof of the previous lemma) $z_k$ must converge to $v$.
Now,
\begin{align*}
\alpha(v;\tilde\otimes_{j=1}^n E_j) \leq \alpha(v;\tilde\otimes_{j=1}^n W_j) & \leq  \alpha(v - z_k;\tilde\otimes_{j=1}^n W_j) + \alpha(z_k;\tilde\otimes_{j=1}^n W_j) \\
& = \alpha(v - z_k;\tilde\otimes_{j=1}^n W_j) + \alpha(z_k;\tilde\otimes_{j=1}^n E_j).
\end{align*}
Since the first term is arbitrarily small as $k$ goes to infinity and the second one converges to $\alpha(v;\tilde\otimes_{j=1}^n E_j)$ we obtain that $\alpha(v;\tilde\otimes_{j=1}^n W_j)=\alpha(v;\tilde\otimes_{j=1}^n E_j)$.
Therefore, $S\overset 1 \hookrightarrow(\tilde\otimes_{j=1}^nW_j,\alpha)$ and this concludes the proof.
\end{proof}

Finally we demonstrate Theorem \ref{RN para tensores}.

\begin{proof}{(of Theorem \ref{RN para tensores}) }

Let us prove $(1) \Rightarrow (2)$ i.e., if $E_1,\dots,E_n,F$ are Asplund spaces, then $(E_1\tilde\otimes\dots\tilde\otimes E_n\tilde\otimes F,\alpha')$ is Asplund.

Take $S\subseteq(E_1\tilde\otimes\dots\tilde\otimes E_n\tilde\otimes F,\alpha')$ a separable subspace. Let us see that $S'$ is separable too. By the previous lemma, there are separable subspaces  $W_j\subseteq E_j$ for $1\leq j\leq n$ and $W_{n+1}\subseteq F$ such that $S\overset 1 \hookrightarrow\big{(}\tilde\otimes_{j=1}^{n+1}W_j, \alpha' \big{)}$. Now, since $E_j$ and $F$ are Asplund, the subspace $W_j'$ is separable for each $j$ and applying Proposition~\ref{separabilidad}, we know that $\U(W_1,\dots,W_n;W_{n+1}')$ is separable. The representation theorem for maximal ideals (see \cite[Theorem 4.5]{FloHun02}) tell us that $\U(W_1,\dots,W_n;W_{n+1}')\overset 1 =(\tilde\otimes_{j=1}^{n+1}W_j;\alpha')'$,
%\[(\tilde\otimes_{j=1}^{n+1} W_j',\alpha )\overset 1 \twoheadrightarrow(\tilde\otimes_{j=1}^{n+1}W_j;\alpha')'\overset 1 = \U(W_1,\dots,W_n;W_{n+1}').\]
 which implies that $(\tilde\otimes_{j=1}^{n+1}W_j,\alpha')'$ is separable and $S'$ is separable too since $(\tilde\otimes_{j=1}^{n+1}W_j,\alpha')'\overset 1 \twoheadrightarrow S'$.

The other implications follow easily using the representation theorem for maximal ideals and the fact that a Banach space is Asplund if and only if its dual has the Radon-Nikodým property.
\end{proof}

%-----------------------------------------------------------------------------------------------------------------------------------------------------

\section{Applications and examples}\label{applications}

In this section we apply the previous results to some classical multilinear ideals. We start with the ideal of extendible multilinear operators $\E$. Recall that a multilinear operator $T :E_1\times\dots\times E_n\rightarrow F$ is \emph{extendible}
if for any Banach space $G_j$ containing $E_j$ there exists
$\widetilde T\in \lg(G_1,\dots,G_n;F)$ an extension of $T$. The space of all such multilinear operators is denoted by $\E(E_1,\dots,E_n;F)$ and it is endowed with the norm $\Vert T\Vert _{\E(E_1,\dots,E_n;F)}$ defined by
$$\begin{array}{rl}
\|T\|_\E:=\inf \{c>0:  & \mbox{ for
all }G_j\supset E_j
\mbox{ there is an extension of }T\mbox{ to }G_1\times\dots\times G_n  \\
& \mbox{ with norm} \leq c\}.
\end{array}$$

In order to use our previous results, we need a proposition first.

\begin{proposition}\label{lemma extendible}
The ideal $\E$ is extendible and $F'$-Arens stable for every dual space $F'$. In addition, if $G$ is a Banach space which contains no copy of $c_0$, then $\E$ has the $G$-RNp.
\end{proposition}
\begin{proof}
By definition, it is trivial to see that this ideal is extendible. It is also well known that $\E$ is $F'$-Arens stable (this can be obtained following the proof of \cite[Theorem 3.6]{Car99} and the fact that every dual space is complemented in its bidual).

Since $c_0$ is an $\lg_\infty$-space, we have $\E(^nc_0;G)=\lg(^nc_0;G)$ and $\E^{min} (^nc_0;G)=\lg_{app}(^nc_0;G)$, the result now follows from Proposition~\ref{teo cociente} and Proposition~\ref{bogda}.
\end{proof}

If we call $\alpha_{ext}$ the tensor norm associated to $\E$, we obtain the next corollary.

\begin{corollary}\label{bases en extendibles}
$ $
\begin{itemize}
\item[(1)] If $E_1, \dots, E_n$ are Asplund spaces and $F'$ is a dual space which contains no copy of $c_0$, then
$$ (E_1' \widetilde\otimes \dots \widetilde\otimes E_n' \widetilde\otimes F', \alpha_{ext}) \overset 1 \twoheadrightarrow \E(E_1,\dots,E_n;F').$$
In particular, $ \E^{min}(E_1,\dots,E_n;F')\overset 1 = \E(E_1,\dots,E_n;F').$
\item[(2)] If $F'$ is a dual space which contains no copy of $c_0$ and $E_1',\dots,E_n', F'$ have  bases $(e_{j_1}')_{j_1},\dots, (e_{j_n}')_{j_n}, (f_l')_l$ respectively, then the monomials
         \[\big(\ e_{j_1}'(\ \cdot \ )\cdots e_{j_n}'(\ \cdot \ )\cdot f_l'\ \big)_{j_1,\dots,j_n,l}\]
         with the square ordering form a Schauder basis of $\E(E_1,\dots,E_n;F')$.
\item[(3)] If $F'$ is a dual space which contains no copy of $c_0$ and $E_1,\dots,E_n, F$ have shrinking Schauder bases $(e_{j_1})_{j_1},\dots, (e_{j_n})_{j_n}, (f_l)_l$ respectively, then the monomials (associated to the coordinate functionals)
         \[\big(\ e_{j_1}'(\ \cdot \ )\cdots e_{j_n}'(\ \cdot \ )\cdot f_l'\ \big)_{j_1,\dots,j_n,l}\]
         with the square ordering form a boundedly complete Schauder basis of the space $\E(E_1,\dots,E_n;F')$.
\end{itemize}

\end{corollary}

\begin{proof}
Item $(1)$ can be deduced using the Proposition \ref{lemma extendible} and Theorem \ref{Lewis theorem}.
Item $(2)$ follows from the first part of Theorem~\ref{bases ideales}. To prove item $(3)$ first note that since the range space is a dual space, by the representation theorem of maximal ideals~\cite[Theorem 4.5]{FloHun02}, we have
 \[\E(E_1,\dots,E_n;F')\overset 1 = \E^{max}(E_1,\dots,E_n;F')\overset 1 =(E_1\tilde\otimes\dots\tilde\otimes E_n\tilde\otimes F,\alpha_{ext}')'.\]
 Now, since the spaces have shrinking bases, then $E_1',\dots,E_n'$ and $F'$ have the bounded approximation property and by the representation theorem of minimal ideals~\ref{representation theorem minimal}, using the first part of this corollary, we have  that
 \[\E(E_1,\dots,E_n;F')\overset 1 = \E^{min}(E_1,\dots,E_n;F')\overset 1 = (E_1'\tilde\otimes\dots\tilde\otimes E_n'\tilde\otimes F', \alpha_{ext}).\]
 In conclusion, we get $(E_1\tilde\otimes\dots\tilde\otimes E_n\tilde\otimes F,\alpha_{ext}')'\overset 1 = (E_1'\tilde\otimes\dots\tilde\otimes E_n'\tilde\otimes F', \alpha_{ext}).$ Notice that the basis $(e_{j_1}'\otimes\dots\otimes e_{j_n}'\otimes f_l')_{j_1,\dots,j_n,l}$ of $(E_1\tilde\otimes\dots\tilde\otimes E_n\tilde\otimes F,\alpha_{ext}')'$ is formed by the coordinate functionals associated to the basis $(e_{j_1}\otimes\dots\otimes e_{j_n}\otimes f_l)_{j_1,\dots,j_n,l}$ of $(E_1\tilde\otimes\dots\tilde\otimes E_n\tilde\otimes F,\alpha_{ext}')$, therefore the result follows.
\end{proof}

Using Theorem \ref{teo achicante} and reasoning in a similar manner, we get:

\begin{corollary}\label{coro monomials}
If $E_1,\dots,E_n$ have shrinking bases, $F''$ has the bounded approximation property and contains no copy of $c_0$, then
\[(E_1'\tilde\otimes\dots\tilde\otimes E_n'\tilde\otimes F,\alpha_{ext})\overset 1 =\E^{min}(E_1,\dots,E_n;F)\overset 1 =\E(E_1,\dots,E_n;F).\]
In particular, if $F$ has also a basis then the monomials form a basis of $\E(E_1,\dots,E_n;F)$.
\end{corollary}

A natural and important question about an ideal is if it preserves some Banach space
property. The next two results shed some light in this direction. The first one is a consequence of Proposition \ref{separabilidad}.

\begin{corollary}\label{coro separa}
Let $E_1, \dots, E_n, F$ be Banach spaces with separable duals.
Then the space $\E(E_1,\dots,E_n;F')$ is separable.
\end{corollary}

Bearing in mind that $\E^{max}(E_1,\dots,E_n;F')=\E(E_1,\dots,E_n;F')$ (since the range space is a dual space), the following corollary can be obtained using Proposition \ref{lemma extendible} and Theorem~\ref{RN para tensores}.

\begin{corollary}\label{coro asplund}
Let $E_1, \dots, E_n,F$ be Banach spaces. The following are equivalent:
\begin{itemize}
\item [(1)] The spaces $E_1,\dots, E_n, F$ are Asplund.
\item [(2)] The space  $(E_1\tilde\otimes\dots\tilde\otimes E_n\tilde\otimes F,\alpha_{ext}')$ is Asplund.
\item[(3)] The space $\E(E_1,\dots ,  E_n, F')$ has the Radon-Nikodým property.
\end{itemize}
\end{corollary}

Now we deal with the ideal of Pietsch-integral multilinear operators.  Recall that a  multilinear operator $T\in \lg(E_1,\dots,E_n;F)$ is \emph{Pietsch integral (Grothendieck integral)} if there exists a regular $F$-valued ($F''$-valued) Borel
measure $\mu$, of bounded variation on $(B_{E_1'}\times\dots\times B_{E_n'}, w^*)$  such that
$$ T(x_1,\dots,x_n) =\int_{B_{E_1'}\times\dots\times B_{E_n'}} (x_1'(x_1))\cdots (x_n'(x_n)) \ d\mu(x_1',\dots,x_n')$$ for every $x_k\in E_k.$
The spaces of Pietsch integral and Grothendieck integral $n$-linear operators  are
denoted by $P\I(E_1,\dots,E_n;F)$ and $G\I(E_1,\dots,E_n;F)$ respectively and the integral norm of a multilinear operator $T$ is defined as $\inf\{ \|\mu\| \}$, where the infimum runs over all the measures $\mu$ representing $T$.

We now describe the minimal kernel of $P\I$ namely, the space of nuclear multilinear operators $\Nu$. A multilinear operator $T\in\lg(E_1,\dots,E_n;F)$ is  \emph{nuclear} if it can be written as $T(x_1,\dots,x_n)=\sum_{j \in \N} \lambda_j
(x_1^j)'(x_1)\cdots(x_n^j)'(x_n)\cdot f_j,$ where $\lambda_j \in \mathbb{K}$,  $(x_k^j)' \in E_k'$, $f_j\in F$ for
all $j$ and $\sum_{j \in \N} |\lambda_j|\cdot \Vert (x_1^j)' \Vert\cdots \|(x_n^j)'\|\cdot \|f_j\|\ <\infty $. The space of nuclear $n$-linear operators is denoted by $\Nu(E_1,\dots,E_n;F)$ and it is a Banach space under the norm
\[ \Vert T\Vert _{\Nu(E_1,\dots,E_n;F)}=\inf \left\{\sum_{j \in \N} |\lambda_j|\cdot \Vert (x_1^j)' \Vert\cdots \|(x_n^j)'\|\cdot \|f_j\|\,\right\}
,\]
where the infimum is taken over all representations of $T$ as
above.

The isometry between the space of Pietsch integral and the space of nuclear multilinear forms on Asplund spaces can be found in \cite{Al85a}. This result and their consequences relied heavily on the theory of vector measures (e.g., \cite{DieUhl77}). We reprove this statement using a quite different perspective: we deal with a systematic scheme on tensor products instead.

The following proposition will allow us to be in the conditions of Theorem \ref{Lewis theorem}.

\begin{proposition}\label{lemma integral}
The ideal $P\I$ is Arens stable, extendible and has the vector-RNp.
\end{proposition}
\begin{proof}
  It is known that $P\I$ is an extendible Arens stable ideal (see for example \cite[Theorem 2.12]{CarLa04} and \cite[Theorem 5]{CarLa05} for an analogous result on the polynomial setting). Let us see that it has the vector-RNp. Using the representation theorem for minimal ideals \ref{representation theorem minimal} and Proposition~\ref{teo cociente} (Pietsch-integral operators are weakly sequentially continuous) it remains to see that $P\I(^nc_0;F)\overset 1 =\Nu(^nc_0;F)$. Indeed, if $T\in P\I(^nc_0;F)$, then $T$ can be written as

\[T(x_1,\dots,x_n)=\int_{B_{\ell_1}\times\cdots\times B_{\ell_1}} x'_1(x_1)\cdots x'_n(x_n) \ d\Gamma(x_1',\dots,x_n'),\]
  where $\Gamma$ is a regular $F$-valued measure of bounded variation on $(B_{\ell_1}\times\cdots\times B_{\ell_1}, w^*)$.

Now, $x_k'(\cdot)=\sum_{j_k=1}^\infty x_k'(e_{j_k}) e_{j_k}'(\cdot)$, then using the Dominated Convergence Theorem applied to the scalar measure $|\Gamma|$, we have

\begin{align*}
T(x_1,\dots,x_n)  = & \int_{B_{\ell_1}\times\cdots\times B_{\ell_1}} \left(\sum_{j_1=1}^\infty x_1'(e_{j_1})\cdot e_{j_1}'(x_1)\right)\cdots\left(\sum_{j_n=1}^\infty x_1'(e_{j_n})\cdot e_{j_n}'(x_n)\right) \ d\Gamma(x_1',\dots,x_n')\\
    = & \int_{B_{\ell_1}\times\cdots\times B_{\ell_1}}\sum_{j_1\dots,j_n} x_1'(e_{j_1})\cdots x_1'(e_{j_n})\cdot e_{j_1}'(x_1) \cdots e_{j_n}'(x_n) \ d\Gamma(x_1',\dots,x_n')\\
= & \sum_{j_1,\dots,j_n} \underbrace{\left(\int_{B_{\ell_1}\times\cdots\times B_{\ell_1}} x_1'(e_{j_1})\cdots x_1'(e_{j_n})\ d\Gamma(x_1',\dots,x_n')\right)}_{:=A_{j_1,\dots,j_n}}\cdot e_{j_1}'(x_1) \cdots e_{j_n}'(x_n)
\end{align*}
Since,
\begin{align*}
\sum_{j_1,\dots,j_n} \|A_{j_1,\dots ,j_n}\|= & \sum_{j_1,\dots,j_n}\left\|\int_{B_{\ell_1}\times\cdots\times B_{\ell_1}} x_1'(e_{j_1})\cdots x_1'(e_{j_n})\ d\Gamma(x_1',\dots,x_n')\right\|\\
\leq & \sum_{j_1,\dots,j_n}\int_{B_{\ell_1}\times\cdots\times B_{\ell_1}} |x_1'(e_{j_1})|\cdots |x_1'(e_{j_n})|\ d|\Gamma|(x_1',\dots,x_n')\\
= & \int_{B_{\ell_1}\times\cdots\times B_{\ell_1}} \|x_1'\|\cdots \|x_n'\|\ d|\Gamma|(x_1',\dots,x_n')\leq \|T\|_{P\I},
\end{align*}
$T$ belongs to $\Nu(^nc_0;F)$ and $\|T\|_{\Nu}\leq\|T\|_{P\I}$, thus $\Nu(^nc_0,F)\overset 1 = P\I(^nc_0;F)$.
\end{proof}

Recall that $\pi$ is the tensor norm associated to $P\I$. We now recover the main result of \cite{Al85a} and other consequences.

\begin{corollary}\label{bases integrales}
The following holds:
\begin{itemize}
\item[(1)] If $E_1, \dots, E_n$ are Asplund spaces, then
$$ (E_1' \widetilde\otimes \dots \widetilde\otimes E_n' \widetilde\otimes F, \pi) \overset 1 \twoheadrightarrow P\I(E_1,\dots,E_n;F).$$
In particular, $\Nu(E_1,\dots,E_n;F)\overset 1 = (P\I)^{min}(E_1,\dots,E_n;F)\overset 1 = P\I(E_1,\dots,E_n;F).$
\item[(2)] If $E_1',\dots,E_n', F$ have  bases $(e_{j_1}')_{j_1},\dots, (e_{j_n}')_{j_n}, (f_l)_l$ respectively, then the monomials
        \[\big(\ e_{j_1}'(\ \cdot \ )\cdots e_{j_n}'(\ \cdot \ )\cdot f_l\ \big)_{j_1,\dots,j_n,l}\]
         with the square ordering form a Schauder basis of $P\I(E_1,\dots,E_n;F)$.
\item[(3)] If $E_1,\dots,E_n, F$ have shrinking Schauder bases $(e_{j_1})_{j_1},\dots, (e_{j_n})_{j_n}, (f_l)_l$ respectively, then the monomials (associated to the coordinate functionals)
         \[\big(\ e_{j_1}'(\ \cdot \ )\cdots e_{j_n}'(\ \cdot \ )\cdot f_l'\ \big)_{j_1,\dots,j_n,l}\]
         with the square ordering form a boundedly complete Schauder basis of the space $P\I(E_1,\dots,E_n;F')$.
\end{itemize}
\end{corollary}

\begin{proof}
 Item $(1)$ follows from Proposition \ref{lemma integral} and Theorem \ref{Lewis theorem}. Items $(2)$ and $(3)$ are obtained similarly to what was done in Corollary \ref{bases en extendibles}.
  %using the fact that $G\I(E_1,\dots,E_n;F')\overset 1 =P\I(E_1,\dots,E_n;F')$ (since the target space is a dual).
\end{proof}

Looking carefully the proof of Theorem \ref{teo achicante} and using the fact that $G\I(E_1,\dots,E_n;F'')\overset 1 =P\I(E_1,\dots,E_n;F'')$ since the range is a dual space, similar manipulations show the following corollary.

\begin{corollary} \label{coroparagrot}
 If $E_1,\dots,E_n$ have shrinking bases and $F''$ has the bounded approximation property, then
  \[(E_1'\tilde\otimes\dots\tilde\otimes E_n'\tilde\otimes F,\pi)\overset 1 =\Nu(E_1,\dots,E_n;F)\overset 1 = P\I(E_1,\dots,E_n;F) \overset 1 = G\I(E_1,\dots,E_n;F).\]
  In particular, if $F$ has also a basis then the monomials form a basis of $G\I(E_1,\dots,E_n;F)$.
\end{corollary}

Note that here we have a coincidence result for $G\I$ (which is bigger than $P\I$). This result is somewhat stronger than those appearing in the literature for $P\I$.

Also as in Corollary \ref{coro separa} and \ref{coro asplund} we can also deduce the following well-known consequences (see \cite{RueSte82}).

\begin{corollary}
Let $E_1, \dots, E_n, F$ be Banach spaces such that $E_i'$ and $F$ are separable spaces, for all  $1 \leq i \leq n$. Then the space $P\I(E_1,\dots,E_n;F)$ is separable.
\end{corollary}

\begin{corollary}
The spaces $E_1, \dots, E_n, F$ are Asplund spaces if and only if the tensor product $(E_1\tilde\otimes \dots\tilde\otimes E_n\tilde\otimes F, \varepsilon)$ is Asplund.
\end{corollary}

%-------------------------------------------------------------------------------------------------------------------------------------------------------------

\section{Coincidence on ideals of polynomials}\label{seccion polinomial}

In this section we provide coincidence results in the context of vector-valued ideals of homogeneous polynomials. Most of the notions and definition in this context are quite similar as those given for multilinear ideals in Section \ref{preliminares}. We omit some of the necessary background and refer the reader to \cite[Section 7]{Flo01} (and also for the theory of their associated tensor norms). For a complementary reading we refer to \cite{Flo97, FloHun02}.

We need a couple of definitions in order to state the analogous results in the context of vector-valued ideals of homogeneous polynomials.

Let $\Q$ be an ideal of vector-valued $n$-homogeneous polynomials (see definition in \cite[Section 7.2]{Flo01} ) and $\gamma$ be a mixed tensor norm (i.e., it assigns to each pair $(E,F)$ a norm on $(\otimes^{n,s}E) \otimes F$ satisfying the properties given in \cite[Section 7.6]{Flo01}, where $\otimes^{n,s} E$ stands for the symmetric tensor product of $E$). We say that $\Q$ and $\gamma$ are associated (and we write this as $\Q \sim \gamma$) if
$$\Q(^nM;N) \overset{1}{=} \left( (\otimes^{n,s}M') \otimes N , \gamma \right),$$
for every finite dimensional spaces $M$ and $N$.

Let $\Q$ be an ideal of $n$-homogeneous polynomials, recall that the Aron Berner extension $AB: \p(^nE;F)\to\p(^nE'';F'')$ is defined by
 \[AB(p)(x):=EXT(\check p)(x,\dots,x),\]
 where $\check p$ is the symmetric multilinear operator associated to $p$ and $EXT$ stands for the iterated extension to the bidual given by $(Ext_n) \circ \dots \circ (Ext_1)$.

 We say that an homogeneous polynomial ideal $\Q$ is \textit{$F$-Aron Berner stable} if the mapping $AB:\Q(^nE;F)\to \Q(^nE'';F)$
 is well defined and results an isometry. Note that the condition above says that the range of the Aron Berner extension remains in $F$.
If $\Q$ is $F$-Aron Berner stable for every $F$, we just say that $\Q$ is an \emph{Aron Berner stable ideal}. We stress that the same terminology was used in \cite{BoGaPe10}, with a different meaning there.  As in the multilinear case, every  maximal ideal of $n$-homogeneous polynomials is $F'$-Aron Berner stable for every dual space (it can be obtained adapting the proof of \cite[Lemma 2.2]{CarGa11b}), although we will not use this fact.

Analogous to the multilinear definition, an ideal $\Q$ is \textit{extendible} if
for every Banach spaces $E,F$, every superspace $G \supset E$ and every $p \in \Q(^nE;F)$ there exists an extension $\widetilde{p} \in \Q(^nG;F)$ of $p$ with the same $\Q$-norm.
Some examples of extendible ideals are $\p_{P\I}$ and $\p_e$ (the definitions of these ideals are completely similar in this context, see for example \cite[Examples 1.11, 1.12]{CarDiMu09}).

We are now ready to prove the polynomial version of Theorem~\ref{Lewis theorem}. Obviously we state first the polynomial version of the Radon-Nikodým property.
\begin{definition}\label{def RNp polinomial}
Let $\Q \sim \gamma$ be an ideal of $n$-homogeneous polynomials and $F$ be a Banach space. We say that $\Q$ has the \textit{$F$-Radon-Nikodým property} ($F$-RNp) if $$(\widetilde\otimes^{n,s}\ell_1(J)  \widetilde\otimes F, \gamma) \overset 1 \twoheadrightarrow  \Q(^nc_0(J);F),$$
for every  index set $J$.

If $\Q$ has the $F$-RNp for all $F$, we say that $\Q$ has the vector-RNp.
\end{definition}

It should be mentioned that, as in the multilinear case, we have the analogous result of Proposition \ref{teo cociente}.

To translate what we know about multilinear operators to the polynomial context we need to make some observations first.  Let $\U$ be an ideal of multilinear operators which is $F$-Arens stable and let us call $\Psi : \U(c_0(B_{E_1'}),\dots, c_0(B_{E_n'});F)\to\U(E_1,\dots, E_n;F)$ the composition of the downward mappings in the right side of the Figure~\ref{diagrama}. The following proposition describes the mapping $\Psi$ more easily (this will be useful to prove the polynomial version of Theorem \ref{Lewis theorem}).

\begin{proposition}
Let $\U$ be an $F$-Arens stable ideal, then the mapping
$$\Psi :  \U(c_0(B_{E_1'}),\dots, c_0(B_{E_n'});F)\to\U(E_1,\dots, E_n;F)$$ is given by $$\Psi(T)(x_1,\dots,x_n)=EXT(T)(I_{E_1}(x_1),\dots,I_{E_n}(x_n)).$$
\end{proposition}
\begin{proof}
For convenience we prove the result for $n=2$. Let $y'\in F'$ and $x_i\in E_i$ ($i=1,2$). Then,
\begin{align*}
y'(\Psi(T)(x_1, x_2)) & = y'[(Ext_2(\Psi_1(T))(x_1,I_{E_2}(x_2))]\\
&= I_{E_2}(x_2)\big(z_2\mapsto (\overleftarrow{J_F\circ \Psi_1(T)})(x_1, z_2, y')\big)\\
&= I_{E_2}(x_2)\big(z_2\mapsto  y'(\Psi_1(T)(x_1, z_2))\big)\\
&= I_{E_2}(x_2)\big(z_2\mapsto  [I_{E_1}(x_1) (z_1\mapsto (\overleftarrow{J_F\circ T})(z_1, z_2,y'))]\big)\\
&= I_{E_2}(x_2)(z_2\mapsto Ext_1(T)(I_{E_1}(x_1),z_2)(y'))\\
&= I_{E_2}(x_2)(z_2\mapsto (\overleftarrow{J_F\circ Ext_1(T)})(I_{E_1}(x_1),z_2,y')\\
&= y'\big((Ext_2\circ Ext_1(T))(I_{E_1}(x_1),I_{E_2}(x_2))\big),
\end{align*}
which concludes the proof.
\end{proof}

Now, this proposition shows that the diagram
 \begin{equation}\label{diagrama para polinomios}
\xymatrix{ \big( \widetilde{\otimes}_{i=1}^n \ell_1(B_{E_i'})\widetilde{\otimes} F, \alpha \big) \ar[rr]^{\varrho_0} \ar@{->>}[d]^{{\otimes}_{i=1}^{n} Q_{E_i'}{\otimes}Id_{F} }
&  & \U(c_0(B_{E_1'}),\dots,c_0(B_{E_n'});F) \ar@{->>}[d]_{\Psi} \\
(\widetilde{\otimes}_{i=1}^{n} E_i'\widetilde{\otimes} F, \alpha) \ar[rr]^{\varrho_n}
& & \U(E_1,\dots,E_n;F)  \\
}
\end{equation}
conmutes and, by the proof of Proposition~\ref{cociente}, we get that, for $E_1, \dots, E_n$ Asplund spaces, the mapping $\Psi$ is a metric surjection if $\U$ is an $F$-Arens stable extendible ideal.

Now we are ready to give a coincidence result in this frame.

\begin{theorem}
Let $E$ be an Asplund space. If $\Q\sim\gamma$ is an $F$-Aron Berner stable extendible ideal  with the $F$-RNp  then,
\begin{equation*}
(\widetilde\otimes^{n,s}E\widetilde\otimes F, \gamma) \overset 1 \twoheadrightarrow \Q(^nE;F).
\end{equation*}
In particular, $ \Q^{min}(^nE;F)\overset 1  = \Q(^nE;F)$.
\end{theorem}

\begin{proof}
A moment of thought shows that, as in Diagram \eqref{diagrama para polinomios}, the following square commutes

\begin{equation*}
\xymatrix{ \big( \widetilde{\otimes}^{n,s} \ell_1(B_{E'})\widetilde{\otimes} F, \gamma \big) \ar[rr]^{\varrho_0} \ar@{->>}[d]^{{\otimes}^{n,s} Q_{E'}{\otimes}Id_{F} }
&  & \Q(^nc_0(B_{E'});F) \ar@{->>}[d]_{\Psi} \\
(\widetilde{\otimes}^{n,s} E'\widetilde{\otimes} F, \gamma) \ar[rr]^{\varrho_n}
& & \Q(^nE;F),  \\
}
\end{equation*}
where $\Psi(p)(x):=AB(p)(I_E(x))$. The rest now follows mimicking the proofs of Proposition~\ref{cociente} and Theorem~\ref{Lewis theorem}.
\end{proof}

From the previous theorem we can deduce coincidence results for $\p_{e}$ and $\p_{P\I}$. The proofs can be obtained by standard manipulations (copying what was done in the previous section).

\begin{corollary}
If $E$ is an Asplund space and $F'$ is a dual space which contains no copy of $c_0$, then
$$ (\widetilde\otimes^{n,s}E' \widetilde\otimes F', \gamma_{ext}) \overset 1 \twoheadrightarrow \p_e(^nE;F'),$$ where $\gamma_{ext}$ stands for the mixed tensor norm associated to $\p_{e}.$

In particular, $ (\p_e)^{min}(^nE;F')\overset 1 = \p_e(^nE;F').$
\end{corollary}

The following corollary was proved by Carando and Dimant in \cite{CarDi00}, using completely different techniques (arguments with a more geometric flavor).

\begin{corollary}
If $E$ is an Asplund space, then
$$ (\widetilde\otimes^{n,s}E' \widetilde\otimes F, \gamma_{int}) \overset 1 \twoheadrightarrow \p_{P\I}(^nE;F),$$
where $\gamma_{int}$ stands for the mixed tensor norm associated to $\p_{P\I}$.

In particular, $\p_{\Nu}(^nE;F)\overset 1 = (\p_{P\I})^{min}(^nE;F)\overset 1 = \p_{P\I}(^nE;F).$
\end{corollary}

From this two corollaries we can also deduce the analogous results obtained for bases in Corollaries~\ref{bases en extendibles}  and \ref{bases integrales} for the polynomial ideals $\p_e$ and $\p_{P\I}$.
An important comment is in order. To obtain the corresponding analogous, a result in the lines of \cite{GeLama61} (or \cite{GreRya05})  in the mixed tensor product $\left( (\widetilde{\otimes}^{n,s}E) \widetilde{\otimes} F , \gamma \right)$ is needed. We have not found it in the literature, so we include some details here.
Let $E$ and $F$ be Banach spaces with bases $(e_j)_j$ and $(f_l)_l$ respectively. We define a natural basis in $\left( (\widetilde{\otimes}^{n,s}E) \widetilde{\otimes} F , \gamma \right)$.

For $\alpha = (j_1, \dots, j_n) \in \mathbb{N}^n$ an index of $n$-elements, we denote by $e_{\alpha}^s := S(e_{j_1} \otimes \dots \otimes e_{j_n})$, where $S : \otimes^n E \to \otimes^{n,s} E$ is the classical symmetrization operator (see, for example \cite{GreRya05}). Let $(\alpha,l)\in \mathbb{N}^n \times \mathbb{N}$, we say that a tensor of the form $$e_{\alpha}^s \otimes f_l \in \left( (\widetilde{\otimes}^{n,s}E) \widetilde{\otimes} F , \gamma \right)$$
is a monomial given by the index $(\alpha,l)$.
We mix the orderings defined in \cite{GeLama61,GreRya05} to define an ordering for the monomials to be a basis. In other words, we consider the square ordering for the ordered index sets $\N^n$ (endowed with the ordering given by Grecu and Ryan) and $\N$. More precisely, given two indexes $(\alpha,l), (\beta,k) \in \mathbb{N}^n \times \mathbb{N}$ we say that
$(\alpha,l) < (\beta,k)$ if $\alpha < \beta$ or, if $\alpha = \beta$ and $l>k$, where the ordering in $\N^n$ is the one given in \cite[Section 2]{GreRya05}).
To prove this it must be shown that the projections to the monomials with their respective ordering are uniformly bounded. The result follows by using carefully the techniques of the two articles mentioned previously and \cite[7.6 (b)]{Flo01}.

\section*{Acknowledgments}
We are extremely grateful to V. Dimant for reading this paper carefully. She helped us to improve the presentation and gave us several references we have included.
We also are indebted to the anonymous referee who dedicated their time to make us notice the countless typos and the unexplained notation that the paper had. He/She also helped us to improve substantially  the drafting of this article.
We thank D. Carando for many inspiring conversations on this problem and also on related topics over the years.
We highlight the contribution of J. Gutierrez, who brought to our attention that Proposition \ref{bogda} was known and already appeared in the literature.
Finally we thank S. Lassalle for mentioning the existence of \cite[Theorem 8]{CarLa08}.

\end{document}